\documentclass{article}
\usepackage{amsmath,amssymb,amsthm}
\usepackage[pagewise]{lineno}
 \usepackage{paralist}
\usepackage{graphicx}  
\usepackage{footnote}
\usepackage{algorithm}
\usepackage{algorithmic}
\usepackage{epsfig,subfigure}
\usepackage{url}
\usepackage{color}
\usepackage{soul}
\usepackage{bm,mathtools,multirow}
\usepackage{float}
\usepackage{perpage}
\MakeSorted{figure}
\MakeSorted{table}
\usepackage{caption}
\usepackage[colorlinks=false]{hyperref}

\input{Parameter-Estimation-Notation.sty}	

\usepackage{color}
\usepackage{url}
\usepackage{cleveref}

\newcommand{\inserted}[1]{{  #1}} 
\newcommand{\edited}[1]{{ #1}}

\newtheorem{lemma}{Lemma}[section]
\newtheorem{theorem}{Theorem}[section]
\newtheorem{corollary}{Corollary}[theorem]	
\newtheorem{proposition}{Proposition}[section]
\newtheorem{definition}{Definition}[section]
{\theoremstyle{plain}
  \newtheorem{assumption}{Assumption}
  \newtheorem{remark}{Remark}}

\crefname{assumption}{Assumption}{assumptions}
\crefrangelabelformat{assumption}{(#3#1#4) to~(#5#2#6)}
\crefname{remark}{Remark}{remarks}
\crefrangelabelformat{remark}{(#3#1#4) to~(#5#2#6)}

\newcommand{\matlab}{$\textsc{MATLAB}^{\textrm{\textregistered}}$}

\addtocounter{footnote}{0}
\newcommand{\bfc}{\mathbf{c}}
\newcommand{\bfe}{\mathbf{e}}

\newcommand{\bfA}{\mathbf{A}}
\newcommand{\bfB}{\mathbf{B}}
\newcommand{\bfC}{\mathbf{C}}

\newcommand{\bfE}{\mathbf{E}}
\newcommand{\bfG}{\mathbf{G}}
\newcommand{\bfH}{\mathbf{H}}

\newcommand{\bfL}{\mathbf{L}}

\newcommand{\bfP}{\mathbf{P}}

\newcommand{\bfS}{\mathbf{S}}

\newcommand{\bfU}{\mathbf{U}}
\newcommand{\bfV}{\mathbf{V}}
\newcommand{\bfW}{\mathbf{W}}
\newcommand{\bfX}{\mathbf{X}}
\newcommand{\bfY}{\mathbf{Y}}
\newcommand{\bfZ}{\mathbf{Z}}
\newcommand{\bfSigma}{\mathbf{\Sigma}}

\newcommand{\bfh}{\mathbf{h}}

\newcommand{\wideregparam}{\regparam}
\newcommand{\bfLambda}{\mathbf{\Lambda}}
\newcommand{\bfDelta}{\mathbf{\Delta}}
\newcommand{\bfPhi}{\mathbf{\Phi}}
\newcommand{\bfPsi}{\mathbf{\Psi}}

\title{Learning Spectral Windowing Parameters for Regularization Using Unbiased Predictive Risk and Generalized Cross Validation  Techniques for Multiple Data Sets} 

\author{M. J.  Byrne and R.  A. Renaut}

\begin{document}
\maketitle

\centerline{\scshape Michael J. Byrne}
\medskip
{\footnotesize
 \centerline{School of Mathematical and Statistical Sciences}
   \centerline{Arizona State University}
   \centerline{Tempe, AZ 85287, USA}
}

\medskip

\centerline{\scshape Rosemary A. Renaut$^*$}
\medskip
{\footnotesize
 \centerline{School of Mathematical and Statistical Sciences}
   \centerline{Arizona State University}
   \centerline{Tempe, AZ 85287, USA}
}

\bigskip


\begin{abstract}
During the inversion of discrete linear systems noise in data can be amplified and result in meaningless solutions. To combat this  effect, characteristics of solutions that are considered desirable are mathematically implemented during inversion, which is a process called regularization. The influence of provided prior information  is controlled by non-negative regularization parameter(s). There are a number of methods used to select appropriate regularization parameters, as well as a number of methods used for inversion. New methods of unbiased risk estimation and generalized cross validation are derived for finding spectral windowing regularization parameters. These estimators are extended for finding the regularization parameters when multiple data sets with common system matrices are available. It is demonstrated that spectral windowing  regularization parameters can be learned from these new estimators applied for multiple data and with multiple windows. The results demonstrate that these modified methods, which do not require the use of true data for learning regularization parameters, are effective and efficient, and perform comparably to a learning method based on estimating the parameters using true data. The theoretical developments are validated for the case of two dimensional image deblurring. The results verify that the obtained estimates of spectral windowing regularization parameters can be used effectively on validation data sets that are separate from the training data, \inserted{and do not require known data}. 
\end{abstract}

\section{Introduction} \label{sec:Introduction}
We consider solutions of ill-conditioned linear problems described by 
\begin{align}
\label{eq:Ax = b}
\bfA\xVec \approx \dVec,
\end{align}
where $\bfA \in \mathbb{R}^{\mA \times n}$ with $m \geq n$ and   $\dVec = \bVec + \noiseVec$ is measured, with $\noiseVec$ being a realization of a random vector and $\bVec = \bfA\xTrue$. Even for invertible square matrices  direct matrix inversion when $\bfA$ is ill-conditioned is not recommended due to the noise in the data. Regularization is generally imposed  in which  desired characteristics of a solution are described mathematically and incorporated into the formulation with the aim to produce a more well-posed problem. 

The generalized Tikhonov regularized solution $\xReg$   \cite{Tikh1963} is given by
\begin{align}
\label{eq:TikSol2}
\xVec(\regparam) = \argmin_{\xVec \in \mathbb{R}^n} \left\{\|\bfA\xVec - \dVec\|_2^2 + \regparam^2\|\bfL\xVec\|_2^2\right\}, \quad \regparam > 0, ~ \bfL \in \mathbb{R}^{\mL \times n}.
\end{align}
The scalar $\regparam>0$ is a regularization parameter, and $\bfL$ is a matrix representation of a linear operator. The term $\|\bfL\xVec\|_2^2$ is an example of a penalty function \cite{Vogel:2002}, and $\bfL$ is called the penalty matrix. If $\bfL = \bfeye_n$, the $n \times n$ identity matrix, then regularization via \cref{eq:TikSol2} is called standard or zeroth-order Tikhonov regularization \cite{ABT}. Other standard choices of $\bfL$ include approximations to first and second order derivative operators  \cite{NeumannDCT,Strang1999,Vogel:2002}. 

The quality of $\xReg$ depends on the choice of both $\regparam$ and $\bfL$. There are a number of methods for selecting $\regparam$ when $\bfL$ has been fixed. As an example, the Morozov discrepancy principle ($\mathtt{MDP}$) \cite{Morozov1966}  requires knowledge of the variance of the noise distribution of the data and selects the regularization parameter as the root of a function. The unbiased predictive risk estimator ($\mathtt{UPRE}$) \cite{Mallows1973} also requires knowledge of the variance of the noise distribution of the data and yields the  regularization parameter as the minimizer of a function. The method of  generalized cross validation ($\mathtt{GCV}$) \cite{Wahba1977,Wahba1990}   does not require the noise distribution be known and also yields the  regularization parameter as a minimizer of a function. Some methods do not solve minimization or root-finding problems; for example, the $\mathtt{L-curve}$ method selects a regularization parameter as the value that locates the point of maximum curvature of a function \cite{Hansen1992,HansenOLeary}. As noted, many of these techniques, including the $\mathtt{UPRE}$, $\mathtt{MDP}$, and $\mathtt{GCV}$ methods, have statistical foundations \cite{StatLearning}. 

\inserted{There has been considerable   research in selecting sets of regularization parameters for a pre-selected set of penalty matrices, a process called multi-parameter (MP) regularization \cite{Brezinski2003,ChungEspanol2017,GazzolaNovati2013,LuPereverzev2011,Wood2002}. Approaches to MP regularization using versions of the $\mathtt{L-curve}$ and $\mathtt{MDP}$ methods can be found in \cite{BelgeKilmerMiller2002} and \cite{Wang2012}, respectively. \edited{An} MP $\mathtt{GCV}$ method was also considered in \cite{ModarresiGolub1,ModarresiGolub2}. Windowing,  either in the data domain or the frequency domain, can also be applied to determine multiple regularization parameters; windowing wavelet coefficients was considered in \cite{EasleyLabatePatel,StephanakisKollias}. Examples of windowed regularization in other frequency domains, such as those generated by discrete trigonometric transforms or the singular value decomposition (SVD), have been presented in \cite{ChEaOl:11,ChungKilmerOLeary,KalkeSiltanen}. There is also recent work on learning semi-norms as regularization operators \cite{Holler2020LearningNR}. }

Techniques for utilizing and analyzing multiple data (MD) sets permeate a multitude of scientific fields as diverse as geoscience \cite{GeoscienceML,Zobitz2020EfficientHD}, and   cancer detection \cite{MedicineML}. A comprehensive overview of data-driven approaches to inverse problems can be found in \cite{Arridge2019SolvingIP}, while specific examples of applying multiple data sets for the solution of inverse problems include \cite{Afkham_2021,ChungChungOLeary2011,ChungEspanol2017,HaberTenorio2003,KunischPock2013,TaroudakiOLeary2015,Learning2005}. 

A main contribution of this paper is  demonstrating and validating how the functions associated with the $\mathtt{UPRE}$ and $\mathtt{GCV}$ methods can be modified to handle multiple data sets,  \edited{using both scalar and sets of regularization parameters dependent on spectral windows}. These are chosen as representative methods that either require ($\mathtt{UPRE}$), or do not require ($\mathtt{GCV}$), prior knowledge of the statistics of the noise in the data. There is a distinction between the modified methods and those that would use averages of the parameter selection functions; \inserted{see \cite{Byrne}}. \inserted{Here, the estimators for spectral windowing regularization are derived from first principles. The results show that whereas the $\mathtt{UPRE}$ method easily extends for spectral  windowing regularization,  the $\mathtt{GCV}$ yields an alternative estimator,  as was noted already in \cite{ChEaOl:11}. The result for the $\mathtt{GCV}$ was shown for standard form regularization and is here extended for general form Tikhonov regularization.   The development of these new methods is presented in  \Cref{sec:Methods}, with proofs of main results in \Cref{app:UPRE,app:GCV}.  }\Crefrange{sec:upre}{sec:gcv} pertain to the $\mathtt{UPRE}$ and $\mathtt{GCV}$ methods, respectively.   In addition to developing these modified parameter selection methods, results are  presented regarding their relationship(s) with the original method(s) and the simplified expressions for obtaining the required functions using the generalized singular value decomposition are provided for each function. \inserted{Proofs for the functions and simplified expressions are given in \Crefrange{app:UPRE}{app:DCT}. }The approach using the $\mathtt{UPRE}$ and $\mathtt{GCV}$ estimators is contrasted with a machine learning approach relying on the knowledge of true data in \Cref{sec:Validation}. Numerical results are shown in  \Cref{sec:Validation}.  Conclusions are given in \Cref{sec:Conclusion}. 

\textbf{Significant Contributions:} It is demonstrated through this work how multiple data sets can be used in conjunction with \edited{windowed} regularization using the  formulation introduced in \cite{ChEaOl:11}. \edited{Windowed }parameter versions of the $\mathtt{GCV}$ method for standard Tikhonov regularization were presented in \cite{ChEaOl:11}. \inserted{Here the windowed $\mathtt{GCV}$ is derived from first principles for the generalized Tikhonov regularization. The $\mathtt{UPRE}$ extension is also derived from first principles for standard and generalized Tikhonov regularization.  For non-overlapping spectral windows the  $\mathtt{UPRE}$ estimator is separable in the spectral domain, so that windows can be found for each window independently.  This is not the case for the $\mathtt{GCV}$ estimator. Moreover, the  windowed $\mathtt{UPRE}$ follows immediately from the standard formula for the scalar parameter $\mathtt{UPRE}$ estimate, while the $\mathtt{GCV}$ function is considerably modified if the leave-one-out approach is strictly applied.}  Numerical results for the restoration of $2$D signals\footnote{Results for $1$D problems and all derivations using the $\mathtt{MDP}$ method are given in \cite{Byrne}}, demonstrate that these new windowed regularization parameter estimators can be used for multiple data sets without knowledge of the true data and that their performance competes with a learning approach in which the training stage requires the knowledge of true data. Furthermore, parameters that have been obtained from one set of training images can be used on a separate set of validation (testing) images distinct from the original set, provided that the signal-to-noise ratios are close.  

\subsection{Summary of notation} \label{sec:Notation}
Due to the consideration of multiple data sets as well as the windowed parameter regularization for these data sets, the notation in this paper is extensive and will therefore be briefly summarized. We will use the acronym    ``MD'' in the textual body of the paper and plots to refer to the multi-data generalizations, of these methods and their corresponding functions. \edited{ For example, the MD $\mathtt{UPRE}$ method means the extension of the $\mathtt{UPRE}$ method for multiple data sets, and the notation $\mathtt{win}$ is used to indicate a function that is defined with respect to regularization parameters used for spectral windows.} 
In the formulation, $R$ denotes the number of data sets being considered, which are indexed by $r$. The index $r$ may either be placed as a subscript or a superscript with parenthesis signs; for example, $\bfA^{(r)}$, $\xVec^{(r)}$, and $\dVec^{(r)}$ represent the system matrix, solution, and data, respectively, associated with the $r$th system. The tilde ($\sim$) is placed above matrices, vectors, or functions to denote their use for MD. Specifically, for matrices the tilde indicates the formation of a block diagonal matrix, e.g. $\widetilde{\bfA} = \diag(\bfA^{(1)},\ldots,\bfA^{(R)})$ for $\{\bfA^{(r)}\}_{r=1}^{R}$.   For vectors, the tilde indicates vertical concatenation, e.g. $\widetilde{\dVec}$ is the ordered vertical concatenation of the vectors $\{\dVec^{(r)}\}_{r=1}^{R}$. The functions used in each parameter selection method are indicated using $F$ with subscripts denoting which method is being considered.  For example, $\UBig(\regparam)$ represents the scalar $\mathtt{UPRE}$ function for MD sets. $P_r$ represents the number of windows/parameters being considered for the $r$th data set; the individual windows are indexed by $p$, $p=1:P$.  
For matrices, vectors, and functions being used for windowed regularization, the subscript ``$\mathtt{win}$" is added. Letters are bolded to indicate vectors; the term $\UWinBig(\regparamVec)$ represents the $\mathtt{UPRE}$ function for use in the most general case, where MD are being used for windowed regularization with the parameters contained in the vector $\regparamVec$. $\nullspace(\bfA)$ is used to denote the null space of the matrix $\bfA$. 

\section{Regularized Solutions for Single and Multiple Data Sets}\label{sec:Background}
 We provide first the standard background for determining the solution of the 
 regularized problem \cref{eq:TikSol2} for a given scalar regularization parameter $\regparam$ \cite{ABT,Hansen:98}.  \inserted{We then extend this for the solution of \cref{eq:TikSol2} using  windowing with respect to the spectral domain  \cite{ChEaOl:11}. Finally, we explain how the scalar and windowed solutions provide a formulation for multiple data sets, as introduced in \cite{ChungEspanol2017}. For the windowed solutions we assume that the windows in the spectral domain are specified in advance.  The aim here is to first show that there is a common framework that can be used to express the solution of the single and multiple data sets problems, for both scalar and windowed formulations. Moreover, we present, as needed, the framework using the generalized singular value decomposition \cite{PaigeSau1}, noting that the results simplify for the SVD as is relevant for standard Tikhonov regularization with $\bfL=\bfeye$   \cite{GolubVanLoan2013}.}

\subsection{The Single Parameter Single Data Set Problem}\label{sec:Tiksolution}
Under the full column rank assumption for the augmented matrix that defines the normal equations,  $\nullspace(\bfA) ~ \cap ~ \nullspace(\bfL) = \{\zeroVec\}$, the normal equations solution of \cref{eq:TikSol2} is given by 
\begin{align}\label{eq:normalequationsolution}
\xVec(\regparam) &= \inv{(\trans{\bfA} \bfA + \regparam^2 \trans{\bfL} \bfL)} \trans{\bfA} \dVec 
= \bfA^{\sharp}(\regparam) \dVec, 
\end{align}
where $\bfA^{\sharp}(\regparam)$ is called the \texttt{generalized inverse} matrix \cite{Hansen:98}. 
A compact representation for \cref{eq:normalequationsolution} is obtained using the generalized singular value decomposition \cite{PaigeSau1}. 
\begin{definition}[The Generalized Singular Value Decomposition (GSVD)\label{def.GSVD}] For real matrices $\bfA$ and $\bfL$ of size $\mA \times n$ and $\mL \times n$, respectively, and \edited{assuming the full column rank condition}, 
 the mutual factorizations
\begin{align}
\label{eq:GSVD}
\bfA = \bfU\bfDelta\trans{\bfX}, \quad \text{and} \quad \bfL = \bfV\bfLambda\trans{\bfX},
\end{align}
exist \cite{ABT}. Here  $\bfU$ is an $\mA \times \mA$ orthogonal matrix, $\bfV$ is a $\mL \times \mL$ orthogonal matrix, and $\bfX$ is an $n \times n$ non-singular matrix.  $\bfLambda$ is a \edited{$\mL \times n$ matrix with non-negative diagonal elements in decreasing order on the principal diagonal  $\Lambda_{jj}$, $1\le j\le \min(\mL,n)=q^*$},   and the only elements of the $\mA \times n$ matrix $\bfDelta$ that are possibly non-zero are
\begin{align*}
 0 \leq \Delta_{(1,k+1)} \leq \Delta_{(2,k+2)} \leq \ldots \leq \Delta_{(\min{(m,n)},k+\min{(m,n)})} \leq 1, \quad k = (n-\mA)_+.
\end{align*}
\edited{Here $(x)_+$ is defined to be $x$ for $x>0$ and $0$, otherwise. }
\end{definition}
Equipped with these mutual factorizations, introducing $\bfY$ as the inverse of $\trans{\bfX}$, and setting $\hat{\dVec}=\trans{\bfU}\dVec$, we obtain 
\begin{align}
	\label{eq:GSVD Normal Eq Sol}
	\xVec(\regparam) &= \bfY\inv{\left(\trans{\bfDelta}\bfDelta + \regparam^2 \trans{\bfLambda}\bfLambda\right)}\trans{\bfDelta}\hat{\dVec}.
\end{align}
But now, with the identity $\trans{\bfDelta} = \trans{\bfDelta}\bfDelta\pinv{\bfDelta}$, where $\pinv{\bfA}$ for a matrix $\bfA$ is the pseudoinverse of $\bfA$  \cite{GolubVanLoan2013} 
 and defining $\bfPhi(\regparam) = \inv{\left(\trans{\bfDelta}\bfDelta + \regparam^2 \trans{\bfLambda}\bfLambda\right)}\trans{\bfDelta}\bfDelta$, provides the filter matrix representation of \cref{eq:GSVD Normal Eq Sol} 
\begin{align}
	\label{eq:GSVDfilter}
	\xVec(\regparam) &= \bfY\bfPhi(\regparam) \pinv{\bfDelta}\trans{\bfU}{\dVec}, 
\end{align}
which also defines $\bfA^{\sharp}(\regparam)= \bfY\bfPhi(\regparam) \pinv{\bfDelta}\trans{\bfU}$. 

\inserted{Examining the entries in the diagonal matrix $\bfPhi(\regparam)$, we have
\begin{align}\label{eq:Phientries}
\Phi_{jj} = 
\begin{cases}
0& j=1:\ell \quad  (\delta_j=0) \\
\frac{\delta_j^2}{\delta_j^2+\regparam^2 \lambda_j^2} & j=\ell+1:q^*\\
1 & j = q^*+1:n.
\end{cases}
\end{align}
Here   $\delta_j$ and $\lambda_j$ are defined as follows. 
Let  $\bm{\delta} = \sqrt{\diag(\trans{\bfDelta}\bfDelta)}$, (for the element-wise square root), then $\delta_j=\Delta_{(k+1,k+j)}$, $j=1:q^*$, and we have defined $\delta_j=0$, $j=1:\ell<n$. Likewise, let $\bm{\lambda} = \sqrt{\diag(\trans{\bfLambda}\bfLambda)}$, $\lambda_j=\Lambda_{jj}$, $j=1:q^*$.  Then $\gamma_j=\delta_j/\lambda_j$, $j=1: q^*$ are the generalized singular values.  Notice $\gamma_j=0$ for $\delta_j=0$, and 
due to the opposite ordering of the $\delta_j$ and $\lambda_j$, the $\gamma_j$ are increasing, which is contrary to the standard ordering of the singular values when using the SVD of $\bfA$. In the following discussion we  assume throughout that $m\ge n$, and that $\bfL$ has full rank $q^*$. Therefore $\delta_j=\Delta_{jj}$ and 
$\Phi_{jj} = \frac{\gamma_j^2}{\gamma_j^2+\regparam^2}$, $j=1:n$ is increasing, with potentially $\Phi_{jj}=0$ if $\delta_j=\gamma_j=0$ for $j=1:\ell<n$, corresponding also to  $(\bfPhi(\regparam) \pinv{\bfDelta}\trans{\bfU} \dVec)_j=0$. Note that according to \cref{eq:Phientries} $\Phi_{jj}=1$, $j>q^*$.}  

\subsection{Spectral Windowing for the Single Data Set Problem}\label{sec:Windowed}
A more general approach to regularization replaces the \edited{scalar} regularization parameter by a vector $\regparamVec = \trans{[\regparam_1,\ldots,\regparam_P]}$ \inserted{in which $\regparam_p$ is the regularization parameter for a $p^{\mathtt{th}}$ solution obtained with respect to a $p^{\mathtt{th}}$ spectral window. These windows are defined, following the approach in \cite{ChEaOl:11}, using non-negative weights  which satisfy
\begin{align}
	\label{eq:Weights}
	\sum_{p=1}^{P} \wVec_j^{(p)} = 1, \quad j = 1:n.
\end{align}
We use the index set for a given window $\mathtt{win}^{(p)}$ which is the set of $j$ such that $\wVec_j^{(p)}\ne 0$. The selection of the windows is predefined and is described in \Cref{sec:windows}. 
The diagonal matrices $\bfW^{(p)} = \diag\left(\wVec^{(p)}\right)$,   satisfying $\sum_{p=1}^P \bfW^{(p)} = \bfeye_n$, introduced into \cref{eq:GSVDfilter} yield the windowed solution which is the sum over all $P$ windows
\begin{align}
	\label{eq:Windowed GSVD Solution}
    \xWin(\regparamVec) = \sum_{p=1}^P \bfY \edited{(\bfW^{(p)})^{1/2}}\bfPhi(\regparam_p) (\bfW^{(p)})^{1/2}\pinv{\bfDelta}\trans{\bfU} \dVec =  \bfY \bfPhi_{\mathtt{win}}(\regparamVec)  \pinv{\bfDelta}\trans{\bfU} \dVec.
\end{align}
This defines the symmetric windowed filter matrix and \texttt{windowed generalized inverse} 
\begin{align}\label{eq:winfilter}
    \bfPhi_{\mathtt{win}}(\regparamVec)&=\sum_{p=1}^P(\bfW^{(p)})^{1/2} \bfPhi(\regparam_p)(\bfW^{(p)})^{1/2} = \sum_{p=1}^P \bfPhi_{\mathtt{win}^{(p)}}(\regparam_p)\\ \label{eq:wingeninv}
    \bfA^\sharp_{\mathtt{win}}(\regparamVec)&= \bfY \bfPhi_{\mathtt{win}}(\regparamVec)  \pinv{\bfDelta}\trans{\bfU}. 
\end{align}
Furthermore, we may also write the solution as a sum over solutions for each window
\begin{align}
	\label{eq:Windowed GSVD Solutionsum}
    \xWin(\regparamVec) = \sum_{p=1}^P \bfY \bfPhi_{\mathtt{win}^{(p)}}(\regparam_p)\pinv{\bfDelta}\trans{\bfU} \dVec = \sum_{p=1}^P \bfA^\sharp_{\mathtt{win}^{(p)}}(\regparam_p)\dVec =\sum_{p=1}^P \xVec^{(p)}(\regparam_p).
\end{align}
Defining 
\begin{align}\label{eq:sepxp}\xVec^{(p)}(\regparam_p) = \bfY \bfPhi(\regparam_p)\pinv{\bfDelta} \bfW^{(p)} \hat{\dVec},
\end{align} 
 we notice that $\xVec^{(p)}(\regparam_p)$ depends only on  $\hat{d}_j$ for $j$ in $\mathtt{win}^{(p)}$  
when the windows are non-overlapping $\wVec_j^{(p)}=1$ for $j\in \mathtt{win}^{(p)}$  and zero otherwise. 

}

\subsection{Multiple data sets} \label{sec:Multiple data sets}
\inserted{Suppose now that we have a collection} of data sets $\{\dVec^{(r)}\}_{r=1}^R$ where, for each data set
\begin{align}
	\label{eq:Big vectors}
	\dVec^{(r)} \approx {\bfA^{(r)}}\xVec^{(r)},
\end{align}
analogously to \cref{eq:Ax = b}, $\bfA^{(r)} \in \mathbb{R}^{m_r \times n_r}$ with $m_r \geq n_r$,   $\dVec^{(r)} = \bVec^{(r)} + \noiseVec^{(r)}$, with $\noiseVec^{(r)}$ being a realization of a random vector and $\bVec^{(r)} = \bfA^{(r)}\xTrue^{(r)}$. 
For given regularization parameters $\regparam^{(r)}$ and penalty matrices $\bfL^{(r)}$ of dimension $q_r \times n_r$, Tikhonov regularization can be performed to produce regularized solutions $\xVec(\regparam^{(r)})$ that minimize the \edited{Tkhonov functionals, $\|\bfA^{(r)}\xVec - \dVec^{(r)}\|_2^2 + \left(\regparam^{(r)}\right)^2\|\bfL^{(r)}\xVec\|_2^2$, which is equivalent to  solving \cref{eq:TikSol2} but for each system independently.} 

\subsubsection{Scalar Parameter for Multiple Data Sets}\label{sec:SPMD}
\inserted{Suppose  each of the systems are \textit{similar} under some assumptions to be defined as needed, then it may be reasonable to replace $\regparam^{(r)}$ by a common scalar $\alpha$ and seek to find a suitable $\alpha$ for all data sets which can then be used to find solutions for other data sets with \textit{similar} properties. In this framework we consider the determination of 
\begin{align} 
\widetilde{\xVec}(\regparam) = \argmin_{\xVec \in \mathbb{R}^{N}}\left\{\|\ABig\widetilde{\xVec} - \widetilde{\dVec}\|_2^2 + {\wideregparam}^2\|\widetilde{\bfL}\widetilde{\xVec}\|_2^2\right\},\label{eq:Big functional}
	\end{align}
for given scalar $\wideregparam$, and where we define $N=\sum_{r=1}^R n_r$. Here, with the notation introduced in~\Cref{sec:Notation}, $\widetilde{\xVec}$ and $\dBig$ are the vectors formed by vertically concatenating $\{\widetilde{\xVec}\}_{r=1}^R$ and $\{\dVec^{(r)}\}_{r=1}^R$, respectively, and matrices $\ABig$ and $\widetilde{\bfL}$ are the block diagonal matrices generated from $\bfA^{(r)}$ and $\bfL^{(r)}$. }The advantage of regularizing via \cref{eq:Big functional} is that we only have to select one parameter, $\wideregparam$, instead of $R$ parameters (one for each data set). \inserted{The disadvantage is that we now have to solve a far larger system of equations, and therefore it will be necessary to identify what it means to be \textit{similar}. 

\inserted{First,  we observe that we can immediately write down the solution of \cref{eq:Big functional} using the framework in \Cref{sec:Tiksolution}, yielding the solution equivalent to \cref{eq:normalequationsolution} for the large system of equations}
\begin{align}\label{eq:normalequationsolutionbig}
\widetilde{\xVec}(\regparam) &= \inv{(\trans{\ABig} \ABig + \wideregparam^2 \trans{\widetilde{\bfL}} \widetilde{\bfL})} \trans{\ABig} \dBig 
= {\ABig}^{\sharp}(\wideregparam) \dBig.  
\end{align}
This is a completely separable block diagonal solution, common only through $\wideregparam$. As for the single data set case, it is convenient to introduce the assumption that there is a GSVD for each system, as given in \Cref{Assumption_Decomposition}. }

\begin{assumption}[\inserted{GSVD for matrix pair $\bfA^{(r)}$, $\bfL^{(r)}$} \label{Assumption_Decomposition}]
We assume that there exist matrices $\bfDelta^{(r)} \in \mathbb{R}^{\mA_r \times n_r}$ and $\mathbf{\bfLambda}^{(r)} \in \mathbb{R}^{\mL_r \times n_r}$ such that
$\bfA^{(r)} = \bfU^{(r)}\bfDelta^{(r)}\trans{(\bfX^{(r)})}$ and $\bfL^{(r)} = \bfV^{(r)}\bfLambda^{(r)}\trans{(\bfX^{(r)})}$
for $r = 1:R$, where $\bfU^{(r)}$ and $\bfV^{(r)}$ are orthogonal \inserted{, $\bfX^{(r)}$ is invertible and $\nullspace(\bfA^{(r)}) ~ \cap ~ \nullspace(\bfL^{(r)}) = \{\zeroVec\}$.} 
\end{assumption}
\inserted{So far this imposes no \textit{similarity} between systems but allows the solution to be written in the filtered form, equivalent to \cref{eq:GSVDfilter}, 
\begin{align}
	\label{eq:GSVDfilterbig}
	\widetilde{\xVec}(\wideregparam) &= \widetilde{\bfY}\widetilde{\bfPhi}(\wideregparam) \pinv{\widetilde{\bfDelta}}\trans{\widetilde{\bfU}} \dBig = {\ABig}^{\sharp}(\wideregparam)\dBig, 
\end{align}
Note this also defines ${\ABig}^{\sharp}(\wideregparam)= \widetilde{\bfY}\widetilde{\bfPhi}(\wideregparam) \pinv{\widetilde{\bfDelta}}\trans{\widetilde{\bfU}}$.  
Again the notation from \Cref{sec:Notation} is applied.  \inserted{This extends the result in \cite[Eq. 3.14]{ChungEspanol2017} to the case when $\bfDelta$ need not be invertible.} Strictly speaking we do not need to use the GSVD for these solutions, but it is the use of the spectral information provided by the GSVD, or SVD when the operator $\bfL$ is replaced by the identity,  that  facilitates in the context of our analysis the definition of a windowed solution dependent on the generalized singular values, respectively singular values if $\bfL=\bfeye$.  We observe that $\widetilde{\bfPhi}$ is symmetric  block diagonal with diagonal blocks. 
We reiterate that \cref{eq:normalequationsolutionbig,eq:GSVDfilterbig} are no more than \cref{eq:normalequationsolution,eq:GSVDfilter}, respectively, simply applied for the large system of equations with no assumptions about any common system information.}
\subsubsection{Windowing for Multiple Data Sets}\label{sec:WinMD}
Having provided the spectral windowing for the single  data set, it is now immediate that we could apply windowing to each of the $r$ data sets. Let $\regparamVec^{(r)} = [\regparam^{(r)}_1,\regparam^{(r)}_2,\ldots,\regparam^{(r)}_{P_r}]$ be the $P_r$ regularization parameters used for windowed regularization applied to the $r$th system described by \cref{eq:Big vectors}. Then, as in \cref{eq:Windowed GSVD Solution}, the  independently constructed regularized solutions yield the  filter representation for each solution
\begin{align*}
\xWin^{(r)}(\regparamVec^{(r)})  =  \bfY^{(r)} {\bfPhi}^{(r)}_{\mathtt{win}}(\regparamVec^{(r)})  \pinv{{(\bfDelta^{(r)})}}\trans{(\bfU^{(r)})} \dVec^{(r)}.
\end{align*}    
\inserted{ Here ${\bfPhi}^{(r)}_{\mathtt{win}}$ incorporates the independent window weighting matrices as in \cref{eq:winfilter} and is again symmetric. As in \cref{eq:wingeninv} we could define the appropriate  generalized inverse matrices for each system. 
If each system has its own set of windows,  there are a total of $\sum_{r=1}^{R} P_r$ regularization parameters. Instead we introduce additional assumptions. 
\begin{assumption}[Windows and $\regparamVec$\label{ass:window}]
We assume 
\begin{enumerate}
\item the number of windows  \edited{$P_r$} is the same for each system, i.e $P_r = P$ for all $r = 1:R$ and  $\regparamVec^{(r)}\in \mathbb{R}^P$, $r=1:R$. 
\item $\regparamVec^{(r)}=\regparamVec $, $r=1:R$.
\end{enumerate}
\end{assumption}
By the first statement of \Cref{ass:window} there are $RP$ parameters. \inserted{Now as for the scalar case, the solution of the problems defined for these concatenated systems is equivalent to solving each independently. But the second statement of \Cref{ass:window} reduces the number of unknowns to $P$ and introduces taking advantage of multiple systems to find an optimal vector  $\regparamVec$. 
Note, this still does not imply that the windows need to be the same.  In summary we have  the standard filtered form for the concatenated solution, as in \cref{eq:Windowed GSVD Solution}, 
\begin{align}\label{eq:Bigwinsoln}
{\xBig}_\mathtt{win}(\regparamVec)  =  \widetilde{\bfY} \widetilde{\bfPhi}_{\mathtt{win}}(\regparamVec)  \pinv{\widetilde{\bfDelta}}\trans{\widetilde{\bfU}} \dBig = \widetilde{\bfA}^\sharp_{\mathtt{win}}\dBig,
\end{align} 
in which the diagonal weighting matrices are hidden within $\widetilde{\bfPhi}_{\mathtt{win}}$, and $\widetilde{\bfA}^\sharp_{\mathtt{win}}$ is defined as in \cref{eq:wingeninv}.}}

\section{Parameter selection methods}\label{sec:Methods}
An estimate of scalar $\regparam$ can be found by applying many different criteria for defining what it means for $\regparam$ to be \emph{optimal}, and there are numerous descriptions in the literature \cite{ABT,Hansen:98,Vogel:2002}.  Here, to find $\regparam$  we focus on the method of generalized cross validation ($\mathtt{GCV}$), \cite{Wahba1977,Wahba1990} and the unbiased predictive risk estimator ($\mathtt{UPRE}$), \cite{Mallows1973,Vogel:2002}. While both techniques are statistically based, the $\mathtt{GCV}$ does not rely on any knowledge of the statistical distribution for  $\noiseVec$, but for the $\mathtt{UPRE}$ we need to assume $\bfSigma^{(r)}$, \edited{the covariance matrix for the noise in sample $r$}, is available. For \cref{eq:TikSol2} these methods require the minimization of an objective function  that depends on the scalar $\regparam$. We present the standard $\mathtt{UPRE}$ and $\mathtt{GCV}$ functions to find $\regparam$ for $R=1$ and then derive the extensions of the $\mathtt{UPRE}$ and $\mathtt{GCV}$ functions for the determination of the \textit{optimal}  $\regparamVec$ for the windowed MD  solution \cref{eq:Bigwinsoln}.  The proofs of the main results are provided in \Cref{app:UPRE,app:GCV}, and the derivations which use the GSVD to simplify the terms that are used to calculate the underlying functions to be minimized dependent on $\regparamVec$ are given in \Cref{app:algebra}.

\subsection{The Unbiased Predictive Risk Estimator}\label{sec:upre}
The $\mathtt{UPRE}$ method was developed in 1973 by Mallows and considers the statistical relationship between the regularized residual $\rReg = \bfA\xReg - \dVec$,
 and the predictive error $\pVec(\regparam) = \bfA(\xReg - \xTrue)$. Assuming $\bfSigma=\sigma^2 \bfeye_m$, the  standard $\mathtt{UPRE}$ objective function for Tikhonov regularization  is given by
\begin{align}
	\label{eq:UPRE}
	\U(\regparam) = \frac{1}{\mA}\|\rVec(\regparam)\|_2^2 + \frac{2\noiseSD^2}{\mA}\trace(\A) - \noiseSD^2 
\end{align}
\cite[p.~98]{Vogel:2002} and the optimal $\regparam$ is defined by  
\begin{align}\label{eq:optscalarUPRE}
\regparam_{\mathtt{UPRE}} = \argmin_{\regparam > 0}  \U(\regparam).
\end{align}
In \cref{eq:UPRE} $\A$  is the data resolution matrix 
\begin{align}
    \label{eq:Influence matrix}
    \A = \bfA\inv{\left(\trans{\bfA}\bfA + \regparam^2\trans{\bfL}\bfL\right)}\trans{\bfA} = \bfA \bfA^{\sharp}(\regparam), 
\end{align} 
and $\U(\regparam)$ is an unbiased estimator of the expected value of the predictive risk, $\frac{1}{\mA}\|\pVec(\regparam)\|_2^2$. The derivation of the $\mathtt{UPRE}$ function for the most general case with $P$ windows and $R$ data sets follows the details for the derivation of \cref{eq:UPRE} as given in \cite[p.~98]{Vogel:2002}, under the additional assumption on the data sets that $\noiseVec^{(r)}$, $r=1:R$, are mutually independent. We collect the required assumptions in \Cref{Assumption_System}.
\begin{assumption}[\inserted{The Data Sets}\label{Assumption_System}]
For $r = 1:R$, assume that $\bVec^{(r)} = {\bfA^{(r)}}\xVec^{(r)}$, $\dVec^{(r)} = \bVec^{(r)} + \noiseVec^{(r)}$, and $\noiseVec^{(r)} \sim \mathcal{N}(\zeroVec^{(r)},\bfSigma^{(r)})$ with mutually independent $\noiseVec^{(r)}$. Equivalently, $\noiseVec^{(r)}$ follows a  Gaussian distribution with mean $\zeroVec$ and symmetric positive definite covariance matrix $\bfSigma^{(r)}$.  The vectors $\bVec^{(r)}$, $\dVec^{(r)}$, and $\noiseVec^{(r)}$ are of length $\mA_r$ and $\xVec^{(r)}$ is of length $n_r$.
\end{assumption}
\begin{theorem}[$\mathtt{UPRE}$ function for multiple data sets and multiple windows\label{thm:Main UPRE Result}]
Under \Cref{ass:window,Assumption_System},  
and assuming that each $\AWin^{(r)}(\regparamVec)=\bfA^{(r)}(\AWin^{(r)})^\sharp$, $r=1:R$, is symmetric, then the $\mathtt{UPRE}$ function $\UWinBig(\regparamVec)$ for the data sets $\{\dVec^{(r)}\}_{r=1}^R$ with windows $\{\{\bfW^{(r,p)}\}_{p=1}^P\}_{r=1}^R$ is
\begin{align}
	\label{eq:Big UPRE}
	\UWinBig(\regparamVecBig) = \frac{1}{M} \sum_{r=1}^R m_r\UWin^{(r)}(\regparamVecBig),
\end{align}
where
\begin{align}
	\label{eq:Individual UPRE}
	\UWin^{(r)}(\regparamVec) = \frac{1}{m_r}\|\rWin^{(r)}(\regparamVec)\|_2^2 + \frac{2}{m_r} \trace(\bfSigma^{(r)} \AWin^{(r)}(\regparamVec)) - \frac{1}{m_r} \trace(\bfSigma^{(r)}), 
\end{align}
and  $M = \sum_{r=1}^{R} m_r$.
\end{theorem}
\begin{proof}The proof is given in \Cref{app:UPRE}. \end{proof}
In \Cref{thm:Main UPRE Result} there is no assumption made that the windows used are the same, nor that the systems are of the same size, namely, we do not assume $m_r=m$, $r=1:R$. Moreover, the results are given in terms of the generalized inverse matrix trace terms and the residuals for  each system, without any use of the GSVD for the matrix pairs $\bfA^{(r)}$ and $\bfL^{(r)}$.

 It is immediate that \cref{eq:Big UPRE} does not align with \cref{eq:UPRE}, in particular,   each of the functions for system $r$ are more general than \cref{eq:UPRE}. To relate the two expressions we assume either $\bfSigma^{(r)}=\sigma_r^2 \bfeye_{m_r}$, where $\sigma_r^2$ is the common variance in the noise, or that the Gaussian noise is whitened by applying the whitening, or zero-phase component analysis (ZCA)  \cite{BellSejnowski}, transformation $(\bfSigma^{(r)})^{-1/2}$ to \cref{eq:Big vectors} for each $r$. 
 \begin{assumption}[Mutually Independent Whitened Data \label{ass:whitenoise}]
We assume $\bfSigma^{(r)}=\sigma_r^2 \bfeye_{m_r}$, $r=1:R$. If the  ZCA  transform is applied to whiten the $r^{\mathtt{th}}$ data set we have $\sigma_r^2=1$
 \end{assumption}
\begin{corollary}[$\mathtt{UPRE}$ for mutually independent data\label{cor:Main UPRE Result}]
Under \Cref{ass:whitenoise} 
\begin{align}
	\label{eq:Individual UPREwhite}
	\UWin^{(r)}(\regparamVec) = \frac{1}{m_r}\|\rWin^{(r)}(\regparamVec)\|_2^2 + \frac{2}{m_r} \sigma_r^2\trace( \AWin^{(r)}\regparamVec)) -  \sigma_r^2.
\end{align}
\end{corollary}
\begin{proof}This is immediate by using $\bfSigma^{(r)}=\sigma_r^2 \bfeye_{m_r} $ in \cref{eq:Individual UPRE}. \end{proof}
Finally, if we have systems that are all of the same size, $m_r=m$, the fraction $m_r/M$ simplifies as $1/R$ and we obtain 
\begin{align}
	\label{eq:Big UPRE same system}
	\UWinBig(\regparamVec) = \frac{1}{R} \sum_{r=1}^R \UWin^{(r)}(\regparamVec).
\end{align}

\begin{remark}[Comment on symmetry for the data resolution matrices\label{remark:symm}] It is clear from the proof given in \Cref{app:UPRE} that it is not necessary to assume symmetry in order to evaluate the estimator, this is only needed to combine the two terms. To be more general we obtained the first steps of the proof without the symmetry requirement. Yet, we can also obtain symmetry without using a mutual spectral decomposition for $\bfA$ and $\bfL$. Indeed, returning to the normal equations form \cref{eq:normalequationsolution} for the solution, without windowing, we can introduce the diagonal matrices $\bfW^{(p)}$ directly to define a $p^{\mathrm{th}}$ solution dependent on $\bfW^{(p)}$ and a scalar regularization parameter $\regparam_p$, 
\begin{align}\label{eq:normalequationsolutionwin}
\xVec^{(p)}(\regparam_p) &= (\bfW^{(p)})^{1/2}\inv{(\trans{\bfA} \bfA + \regparam_p^2 \trans{\bfL} \bfL)}(\bfW^{(p)})^{1/2} \trans{\bfA} \dVec 
= \bfA_p^{\sharp}(\regparam_p) \dVec.  
\end{align}
The data resolution matrix associated with \cref{eq:normalequationsolutionwin} is immediately symmetric, and it reduces to the original windowed form for $\xVec^{(p)}(\regparam_p)$ when the GSVD is applied. Still, the more general expression \cref{eq:normalequationsolutionwin} facilitates an alternative direction for finding multiple regularization parameters based on blocks of components in $\trans{\bfA}\dVec$, such as using domain multisplitting \cite{renaut-1998-msls,RLG}. The expression is also relevant when it is not feasible to find a mutual spectral decomposition but it is possible to implement forward operations with matrices $\bfA$ and $\bfL$. Noting that the proof of \Cref{thm:Main UPRE Result} does not make any assumptions about the use of the GSVD, we reiterate that \Cref{thm:Main UPRE Result} provides a $\mathtt{UPRE}$ form that can be used more generally. 
\end{remark}

We now focus on the use of mutual spectral decompositions for $\bfA$ and $\bfL$ that can be used to simplify the expressions used to estimate $\regparamVec$ as the solution of \cref{eq:optscalarUPRE} applied for \cref{eq:Big UPRE}.  The standard approach uses the GSVD (or the SVD when $\bfL=\bfeye$) that is also relevant for Kronecker product forms for $\bfA$ and $\bfL$. Discrete Fourier or cosine transforms ($\mathtt{DCT}$)  also provide a mutual decomposition that can be employed. Here we give the results in terms of the GSVD, and will show how it can be used in \Cref{sec:Validation}.  The proofs of elementary algebraic results are given in \Cref{app:algebra}. The relation of the $\mathtt{DCT}$ to the GSVD is given in \Cref{app:DCT} 

For the single data set case with windowing, \edited{and introducing the diagonal matrix $\bfPsi=\bfeye-\bfPhi$, \Crefrange{lem:normres}{lem:tracewin} immediately provide}
\begin{align}
    \|\rWin(\regparamVec)\|_2^2  &= \sum_{j=1}^{q^*} \left(\sum_{p=1}^Pw_j^{(p)}\Psi_{jj}(\regparam_p)\right)^2  \hat{d}_j^2+ \sum_{j=n+1}^m  \hat{d}_j^2, 
\end{align}
and  
\begin{align}
    \trace( \bfA_{\mathtt{win}}(\regparamVecBig)) &= (n-q^*)+ \sum_{j=\ell+1}^{q^*} \sum_{p=1}^P w_j^{(p)} \Phi_{jj}(\regparam_p),
\end{align}
respectively.
Hence, ignoring constant terms, and using the index $(r,p)$ to indicate data set $r$ and window $p$, we can write 
\begin{multline}\label{eq:UPREBigR}
    \UWinBig(\regparamVec) = \frac{1}{M} \sum_{r=1}^R \left(\sum_{j=1}^{q_r^*} \left(\sum_{p=1}^Pw_j^{(r,p)}\Psi^{(r)}_{jj}(\regparam_p) \right)^2 \left(\hat{d}^{(r)}_j\right)^2+ \right.\\
    \left.
    {2\sigma_r^2} \sum_{j=\ell+1}^{q_r^*} \sum_{p=1}^P w_j^{(r,p)} \Phi^{(r)}_{jj}(\regparam_p) \right).
    \end{multline}
Likewise, for non-overlapping windows using \Cref{lem:normres,lem:tracewin}, ignoring constant terms we obtain a function for each window, 
\begin{align}\label{eq:UPREwinsepR}
    \UWinBig^{(p)}(\regparam_p) = \frac{1}{M} \left(\sum_{r=1}^R \sum_{j\in\mathtt{win}^{(r,p)} }\left(\Psi^{(r)}_{jj}(\regparam_p)  \hat{d}^{(r)}_j\right)^2 + 
    {2\sigma_r^2}\sum_{j\in\mathtt{win}^{(r,p)} }\Phi^{(r)}_{jj}(\regparam_p) \right).
\end{align}

Finally, combining \cref{eq:UPREBigR,eq:UPREwinsepR} we have the main result to estimate the windowed regularization parameters using $R$ measurements $\dVec^{(r)}$ 
for  unknown $\xVec^{(r)}$, $r=1:R$ given the GSVD for the system 
$\bfA$ with penalty matrix $\bfL$. In this case $m_r=m$, $n_r=n$, $q^*_r=q^*$, $M=Rm$, and we assume that the same windows are used for each system. 
\begin{proposition}
Under \Crefrange{ass:window}{ass:whitenoise}, with now $\bfA^{(r)}=\bfA$, $\bfL^{(r)}=\bfL$ and  $\bfSigma=\bfSigma^{(r)}=\sigma^2 \bfeye_m$, the windowed parameter vector $\regparamVec$ for the $\mathtt{UPRE}$ function, after removing the common factor $1/m$  is 
\begin{multline}
	\label{eq:Big UPRELearn}
\regparamVec_{\mathtt{UPRE}}=\argmin_{\regparamVec \in \mathbb{R}_+^P}	\left\{ \frac{1}{R} \sum_{r=1}^R\left(\sum_{j=1}^{q^*} \left(\sum_{p=1}^Pw_j^{(p)}\Psi_{jj}(\regparam_p) \right)^2 \left(\hat{d}^{(r)}_j\right)^2\right) + \right.\\
\left.{2\sigma^2} \sum_{j=\ell+1}^{q^*} \sum_{p=1}^P w_j^{(p)} \Phi_{jj}(\regparam_p)\right\}.
\end{multline}
For non-overlapping windows $\regparamVec$ can be found by minimizing for each scalar parameter $\regparam_p$ independently  using the frequency domain data $\hat{\dVec}^{(r)}$, $r=1:R$
\begin{align}
	\label{eq:Window UPRELearn}
\regparam_p=\argmin_{\regparam >0}\left\{
\frac{1}{R} \sum_{r=1}^R \left(\sum_{j\in\mathtt{win}^{(p)} }\left(\Psi_{jj}(\regparam) \hat{d}^{(r)}_j\right)^2 \right) + 
    {2\sigma^2}\sum_{j\in\mathtt{win}^{(p)} }\Phi_{jj}(\regparam)\right\}.
\end{align}
\end{proposition}

\subsection{Generalized Cross Validation}\label{sec:gcv}
In contrast to the $\mathtt{UPRE}$ function, the $\mathtt{GCV}$ function does not require knowledge of the covariance matrix $\bfSigma$. Yet, it is statistically based, and is obtained by applying a leave-one-out analysis. It  yields the function
\begin{align}
	\label{eq:GCV}
	\G(\regparam) = \frac{\frac{1}{\mA}\|\rReg\|_2^2}{\left(\frac{1}{\mA}\trace(\bfeye_{\mA} - \A)\right)^2} = \frac{\frac{1}{\mA}\|\rReg\|_2^2}{\left(1 - \frac{1}{\mA}\trace(\A)\right)^2}, 
\end{align}
for which we define 
\begin{align}\label{eq:argminsGCV}   \regparam_{\mathtt{GCV}} = \argmin_{\regparam > 0} \G(\regparam).
\end{align}
It is immediate that we can define $\regparam_{\mathtt{GCV}}$ for the concatenated systems, assuming a leave-one-out-analysis applied for the system defined by $\widetilde{\bfA}$, replacing $\mA$ by $M$
and the residual and trace terms calculated for $\widetilde{\bfA}$. On the other hand, it was shown in \cite{ChEaOl:11} that the immediate application of \cref{eq:GCV} for the system described by the single system with windowing is not consistent with the leave-one-out analysis. Rather, deriving the $\mathtt{GCV}$ for the windowed system from first principles, assuming the standard Tikhonov regularization with $\bfL=\bfeye_n$,  yields a new function \cite[Theorem 3.2]{ChEaOl:11} which is more complex algebraically. Applying the same analysis from first principles based on the Allen Press function, \cite{GoHeWa}, for the generalized Tikhonov form, we arrive at the following result which expands on \cite[Theorem 3.2]{ChEaOl:11}. 
\begin{theorem}[Windowed $\mathtt{GCV}$ Function for a single data set\label{thm:WindowedGCV}]
Assume $\bfL$ has full rank (row rank if $\mL<n$ but column rank if $\mL\ge n$), $\lambda_j>0$, $j=1:q^*$,  $\mA\ge n$, and  $0=\delta_1 =\dots= \delta_\ell<\delta_{\ell+1} \dots \le \delta_n$. The windowed $\mathtt{GCV}$ function for the generalized Tikhonov regularization is given by 
\begin{multline}\label{thm:windowgcvfunc}
    \GWin(\regparamVec) =
    \frac{1}{m} \left(\sum_{j=q^*+1}^m \left(1 + \left(\sum_{p=1}^P\frac{1-\nu_p}{\mu_p}\right)\right)^2\hat{d}_j^2 + \right.\\
      \left.\sum_{j=1}^{q^*}\left( 1 +\left(\sum_{p=1}^P\frac{1-\nu_p}{\mu_p}\right) -\left(\sum_{p=1}^P \frac{1}{\mu_p}
     \frac{\gamma_j^2w_j^{(p)}}{\gamma_j^2+\regparam_p^2}\right) \right)^2\hat{d}_j^2  \right).
    \end{multline}
Here the number of windows is $P$, the weights on each window for generalized singular value $\gamma_j$ are given by $w_j^{(p)}$,  and $\mu_p$ and $\nu_p$ are obtained from \Cref{lem:diag} as
\begin{align*}
    \mu_p &= \frac{1}{m}\left(m-n +q^*- \sum_{j=1}^{q^*}\frac{\gamma_j^2}{\gamma_j^2+\regparam_p^2}\right) \text{ and }
    \nu_p= \frac{1}{m}\left(m-n +q^*-  \sum_{j=1}^{q^*}\frac{\gamma_j^2w_j^{(p)}}{\gamma_j^2+\regparam_p^2} \right). 
\end{align*}
\end{theorem}
\noindent Note that $\Phi_{jj}=1$ for $j>q^*$ and so some terms in the summation cancel when $q^*<n$. The same applies for $(\Phi_{\mathrm{win}^{(p)}}(\regparam_p))_{jj}$, also with $w_j^{(p)}=0$.

In contrast to the analysis for the $\mathtt{UPRE}$, the derivation and associated proof, as given in \Cref{app:GCV}, for \Cref{thm:WindowedGCV} relies explicitly on the availability of a mutual spectral decomposition for $\bfA$ and $\bfL$. In  \cite{ChEaOl:11} the equivalent result relied on the use of the SVD and made the assumption that all singular values are positive. The result here explicitly permits $\delta_j=0=\gamma_j$, and is sufficiently general that it applies for the case when  $\bfL$ is tall or wide. 
On the other hand, as noted in \cite{ChEaOl:11}, \cref{thm:windowgcvfunc} is a complicated non-linear function of $P$ variables, and it can be beneficial to consider alternative approximations for finding a good $\mathtt{GCV}$ type estimator for $\regparamVec$. There, it is suggested that non-overlapping windows can be used to find a good initial point for optimization, in which case a $\mathtt{GCV}$ function for each window can be optimized separately for scalar $\regparam_p$. Here we consider this approach, but also compare with an estimator that is obtained by applying the standard $\mathtt{GCV}$ function  in which the residual and trace terms in the numerator and denominator are calculated directly. Having provided \Cref{thm:WindowedGCV} it is immediate, as already noted, that this estimator is not a true $\mathtt{GCV}$ estimator for the windowed case. 

\inserted{First we note that in the scalar parameter case, \cref{eq:GCV} generalizes to
\begin{equation} 
\label{eq:GCV_Big}
    \GBig(\regparam) =  \frac{\frac{1}{M}\left\|\rBig(\regparam)\right\|_2^2}{\left(1 - \frac{1}{M}\trace\left(\ABig(\regparam)\right)\right)^2}
\end{equation}
for the MD case with $M = \sum_{r=1}^R m_r$. When using non-overlapping windows, from \cref{eq:GCV_Big} and $\AWinBig(\regparamVec) = \ABig\AWinBig^{\sharp}$ we can select $\regparam_p$ for $p = 1:P$ using the independent functions
\begin{equation} 
\label{eq:GCV_Big Decoupled}
	\GEstWinBig^{(p)}(\regparam_p) = \frac{\frac{1}{M}\|\widehat{\rBig}_{\mathtt{win}}^{(p)}(\regparam_p)\|_2^2}{\left(1 - \frac{1}{M}\trace\left(\ABig\ABig_p^{\sharp}\right)\right)^2},
\end{equation}
with $\widehat{\rBig}_{\mathtt{win}}^{(p)}(\regparam_p) = \trans{\widetilde{\bfU}}\ABig \xBig^{(p)}(\regparam_p) - \widetilde{\bfW}^{(p)} \trans{\widetilde{\bfU}} \dBig$. \Cref{eq:GCV_Big Decoupled} can also be applied to a single data set $\dVec \in \mathbb{R}^m$ and non-overlapping windows (the tilde notation is dropped and $M$ replaced by $m$).
The use of \cref{eq:GCV_Big Decoupled} is analogous to \cite[eq. 3.13]{ChungEasleyOLeary} extended to the MD environment; \cref{eq:GCV_Big Decoupled} is not a true \texttt{GCV} function, though minimization is easier than minimizing \cref{thm:windowgcvfunc}. As with the derivations of the $\mathtt{UPRE}$ functions in \cref{sec:upre}, \cref{eq:GCV_Big Decoupled} can also be written in terms of the GSVD using the results in \Cref{app:algebra}.} Finally, we note, as with the $\mathtt{UPRE}$, that having estimators for the scalar and windowed parameter solutions immediately provides the estimators for the block  systems of equations for multiple data sets. 



\section{Numerical Experiments}\label{sec:Validation}
To evaluate the effectiveness of the spectral windowing and scalar parameter selection methods described in \Cref{sec:Methods}, \inserted{we present the results for a single two dimensional test problem, with different noise levels and blur width. Results for one dimensional problems are presented in \cite{Byrne}}. For both problems the parameter(s) are found for a set of training data, and then validated using a separate validation  set.  The results for the $1$D problem, which serves as a proof of concept for the methods and that use MRI data that is built into \matlab, are detailed in \cite{Byrne}. More relevant are the results for the two-dimensional problem which is described in detail here. 
This problem utilizes the images in  \cref{fig:MESSENGER True} of the planet Mercury obtained by the MESSENGER space probe\footnote{The selected MESSENGER images are available to the public courtesy of NASA and JPL-Caltech \cite{MESSENGER}. The identifiers of the images in \cref{fig:MESSENGER True}, starting from the top row and moving left to right, are: PIA10173, PIA10174, PIA10177, PIA10942, PIA11246, PIA12042, PIA12068, PIA12116, PIA14189, PIA15756, PIA18372, PIA19024, PIA19203, PIA19213, PIA19267.}.
\begin{figure}[htbp]
\centering
\includegraphics[width=0.9\textwidth]{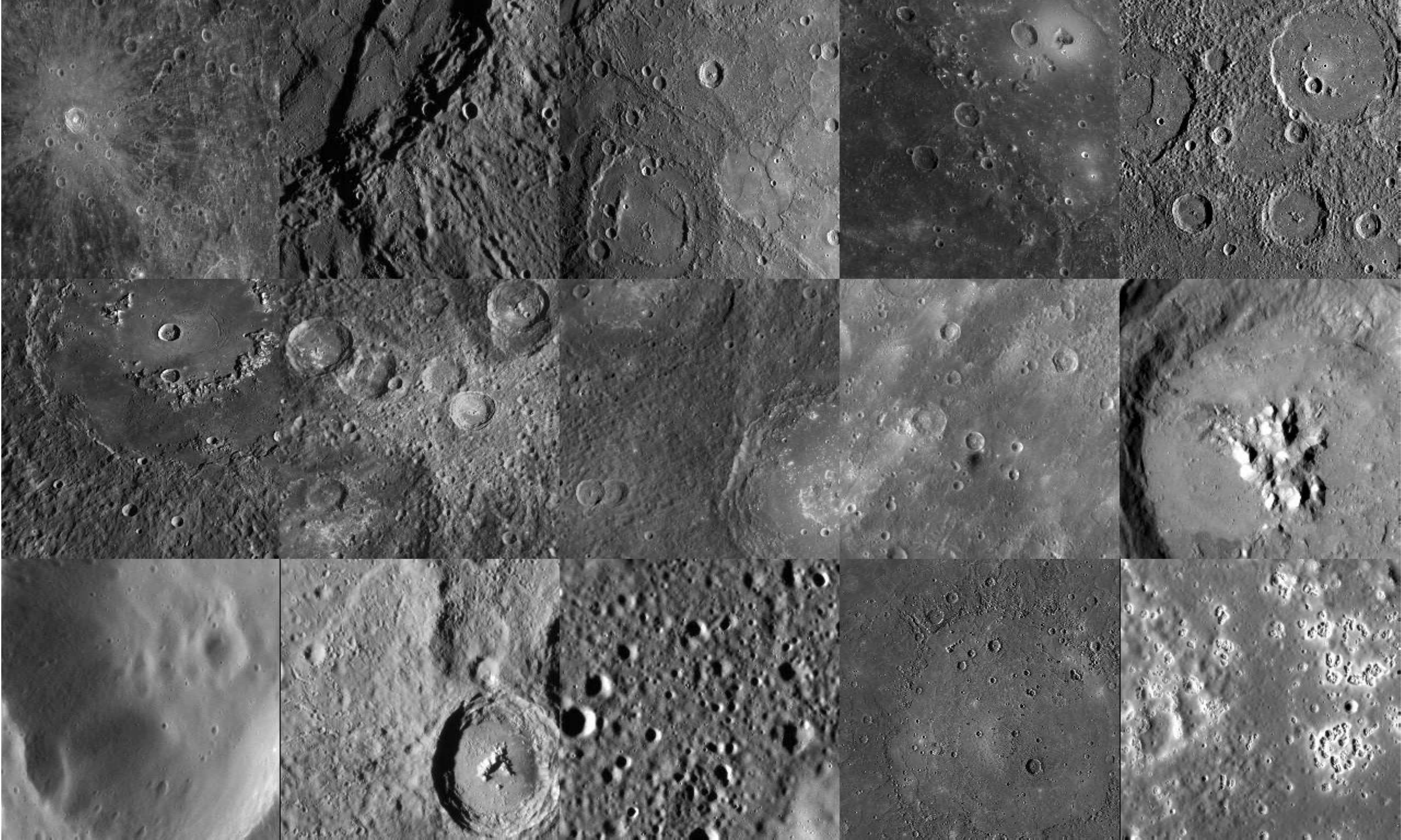}
\caption{Selected images used for MESSENGER 2D test problem. Available courtesy of NASA/JPL-Caltech \cite{MESSENGER}.}
\label{fig:MESSENGER True}
\end{figure}
The signal-to-noise ratio (SNR) is used as a measurement for noise content in the images and is given by
\begin{align*}
\text{SNR} = 10\log_{10}\left(\frac{\mathcal{P}_{\text{signal}}}{\mathcal{P}_{\text{noise}}}\right).
\end{align*}
In the discrete setting, the average power $\mathcal{P}$ of a vector $\xVec$ of length $n$ is defined as $\frac{1}{\mA}\|\xVec\|^2_2$. Using this definition for vectors $\bVec$ and $\noiseVec$, $\mathcal{P}_{\text{signal}} = \frac{1}{\mA}\|\bVec\|^2_2$ and $\mathcal{P}_{\text{noise}} = \frac{1}{\mA}\|\noiseVec\|^2_2$ and so the quotient in the logarithm is $\|\bVec\|_2^2/\|\noiseVec\|_2^2$. If $\bVec$ is a matrix representing an image, in which case $\noiseVec$ is a realization of a random matrix, the $2-$norm can be replaced by the Frobenius norm. 

For a basis of comparison, parameters were also selected as minimizers of the \textit{learning} function
\begin{align}
	\label{eq:Learning function}
	\BWinBig(\regparamVec) = \frac{1}{R}\left\|\xWinBig(\regparamVec) - \xBig\right\|_2^2 = \frac{1}{R}\sum_{r=1}^R \BWin^{(r)}(\regparamVec),
\end{align} 
where
\begin{align}
	\label{eq:Individial Best Function}
    \BWin^{(r)}(\regparamVec) = \|\xWin^{(r)}(\regparamVec) - \xVec^{(r)}\|_2^2.
\end{align}
Note that this definition requires that the true solutions, $\{\xVec^{(r)}\}_{r=1}^R$,  are known, \edited{as it finds the parameters to minimize the mean squared relative error ($\mathtt{MSE}$) using known data}. Regularization parameters chosen as minimizers of \cref{eq:Learning function} are \textit{optimal} in the sense of minimizing the \edited{$\mathtt{MSE}$} of the regularized solutions $\xWin^{(r)}(\regparamVec)$; the use of \cref{eq:Learning function} was considered in \cite{ChungEspanol2017}. One could also find minimizers $\regparamVec^{(r)}$ of \cref{eq:Individial Best Function} for each $r = 1:R$, which would produce parameters that are \textit{optimal} for their own data set. In the results we use \edited{$\mathtt{MSE}$} to indicate results that are found using the learning function \cref{eq:Individial Best Function}.

In the experiments we use the spectral windows as described in \Cref{sec:windows} and to evaluate the forward operators we use the Kronecker product $2$D discrete cosine transform ($\mathtt{DCT}$). The relation of the $1$D $\mathtt{DCT}$ to the GSVD is  briefly described in \Cref{sec:Decomp}. The extension that relates the Kronecker product $2$D $\mathtt{DCT}$ to the KP GSVD follows similarly. 
\subsection{Spectral Windows}\label{sec:windows}
\edited{In the experiments we use windowing following the approach in \cite{ChEaOl:11}. }
 First we consider non-overlapping windows,  $\bfW^{(p)}$, for which the components of their corresponding weight vectors $\wVec^{(p)}$ satisfy
\begin{align}
	\label{eq:Non-overlapping window condition}
    \wVec_j^{(p)} \in \{0,1\}, \qquad j = 1:n, \quad p = 1: P.
\end{align}
The condition given by \cref{eq:Non-overlapping window condition} means that for each $j = 1:n$, there is exactly one $p \in \{1:P\}$ such that $\wVec_j^{(p)} = 1$.  

Perhaps the simplest way of choosing the components of $\wVec^{(p)}$ is to first choose $P+1$ partition values $\omega^{(0)} \geq \ldots \geq \omega^{(P)}$ such that $\omega^{(0)} \geq \singular_1$ and $\singular_n > \omega^{(P)}$, then  for $p = 1:P$, the  non-zero components occur for those $j$ for which $\singular_j$ is in the $p^{\mathtt{th}}$ window: 
\begin{align}
	\label{eq:Non-overlapping windows}
	\wVec_j^{(p)} =  
	1, \text{ for } \omega^{(p-1)} \geq \singular_j > \omega^{(p)}.
\end{align}
\inserted{Here $\singular_j$ correspond to the generalized singular values, ordered in the opposite order, decreasing rather than increasing. This is consistent with the standard ordering of the singular values when using the SVD of the matrix $\bfA$.} \inserted{The partition values $\partition^{(0)} \geq \ldots \geq \partition^{(P)}$ used in the experiments are formed by taking $P+1$ linearly or logarithmically equispaced samples of the interval $[\gamma_{1},\gamma^*_{n}]$, where $\gamma^*_{n}$ is a largest non-infinite generalized singular value of $(\bfA,\bfL)$.} Partition values can also be used to generate overlapping windows. For example, cosine windows are defined in \cite[Eq. 3.6-3.7]{ChEaOl:11} by using midpoints of each partition. \inserted{Linearly and logarithmically spaced cosine windows are used in the experiments as examples of overlapping windows.}

\subsection{The \texorpdfstring{$\mathtt{DCT}$}{DCT}\label{sec:Decomp}} 
While the GSVD is useful for analyzing problems with a general matrix $\bfA$, for practical image deblurring problems with $m_r=n_r$ it is more computationally efficient to use the 
 $2$D discrete cosine transform ($\mathtt{DCT}$). 
 Assuming that reflexive boundary conditions are applied, \inserted{primarily to reduce the potential for reflection that would arise with zero boundary conditions}, then both $\bfA$ and $\bfL$ have the same block structure and the $\mathtt{DCT}$ can be used to simultaneously diagonalize $\bfA$ and $\bfL$ into  a 
BTTB + BTHB + BHTB + BHHB matrix, with the ``T" and ``H" standing for Toeplitz and Hankel, respectively, \cite{HansenNagyOLeary}.  \Cref{thm:Diag to GSVD}, for which a brief proof is given in \Cref{app:DCT}, describes how a simultaneous diagonalization of these matrices is related to their GSVD.
\begin{theorem}{GSVD for the $\mathtt{DCT}$ \cite{Byrne}
\label{thm:Diag to GSVD}}
Given the simultaneous diagonalization of the $n \times n$ symmetric matrices $\bfA = \trans{\bfC}\widetilde{\bfDelta}{\bfC}$ and $\bfL = \trans{\bfC}\widetilde{\bfLambda}{\bfC}$ where $\bfC$ is orthogonal and $\nullspace(\bfA) ~\cap~ \nullspace(\bfL) = \{\zeroVec\}$, $\bfA$ and $\bfL$ can be expressed as $\bfA = \bfU\bfDelta\trans{\bfX}$ and $\bfL = \bfU\bfLambda\trans{\bfX}$, respectively, where $\bfU$ is orthogonal, $\bfX$ is invertible, $\trans{\bfDelta}\bfDelta + \trans{\bfLambda}\bfLambda = \bfeye_n$, 
$
	0 \leq \Delta_{1,1} \leq \ldots \leq \Delta_{n,n} \leq 1,$ and $
	1 \geq \Lambda_{1,1} \geq \ldots \geq \Lambda_{n,n} \geq 0.
$
\end{theorem}

\inserted{\Cref{thm:Diag to GSVD} serves as a theoretical tool that bridges different matrix decompositions for the purpose of forming a greater variety of implementation options. Specifically, this result assists in the efficient implementation for the regularization estimators, both $\mathtt{UPRE}$ and $\mathtt{GCV}$, using the simplified expressions for the results in \Crefrange{app:UPRE}{app:algebra}. While \Cref{thm:Diag to GSVD} applies to any simultaneous diagonalization of symmetric matrices, the $\mathtt{DCT}$ has the advantage of avoiding complex operations that would generally arise if using fast Fourier transforms, rather than the $\mathtt{DCT}$. Naturally, we could also impose zero boundary conditions or periodic boundary conditions \cite{HansenNagyOLeary}. Our choice to use the $\mathtt{DCT}$ is not generally limiting but appropriate for the given application.}

\subsection{Two-dimensional problems}
The data sets for the 2D test problem consist of images of size $256 \times 256$. \inserted{A total of $16$ images were used and split into training and validation sets containing $8$ images each. To obtain the $16$ images from the $512 \times 512$ Mercury images in \cref{fig:MESSENGER True}, the first $8$ images were chosen and two $256 \times 256$ subimages of each image were selected as the northwest and southeast corners.} Another validation set of images, shown in \cref{fig:Validation Set 2}, was used that consisted of built-in \matlab~images.
\begin{figure}[htbp]\centering
\includegraphics[width=0.9\textwidth]{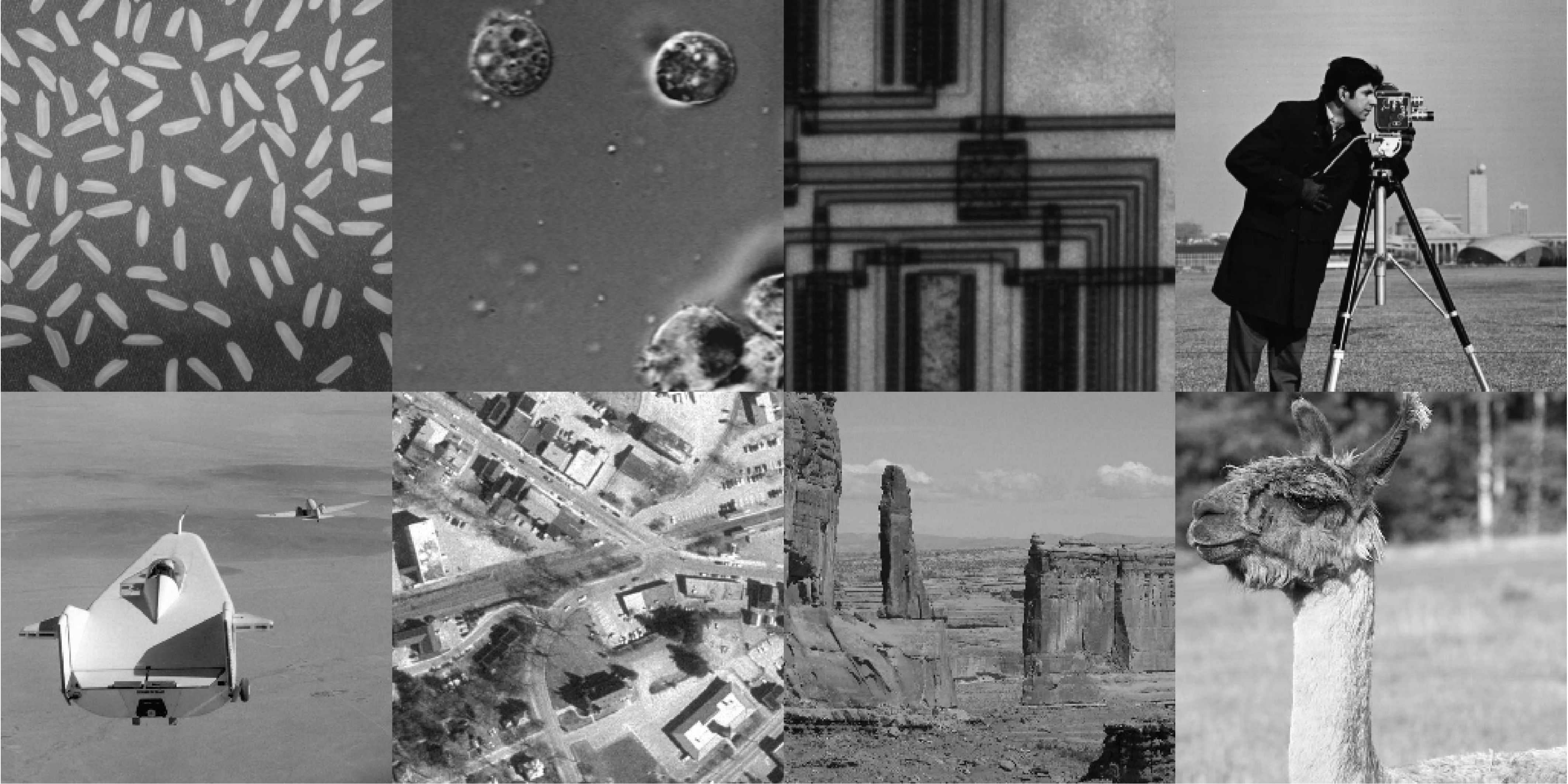}
\caption{The second validation set, consisting of built-in \matlab~images. From left to right starting in the top row, the images are: \texttt{rice.png}, \texttt{AT3\_1m4\_01.tif},~\texttt{circuit.tif}, \texttt{cameraman.tif},~\texttt{liftingbody.png},~\texttt{westconcordorthophoto.png}, \texttt{parkavenue.jpg}, and \texttt{llama.jpg}.}
\label{fig:Validation Set 2}
\end{figure}

A $256 \times 256$ point spread function was formed using a discretization of the zero centered, circularly symmetric Gaussian kernel, $k(x,y) = \exp\left(-{(x^2 + y^2)}/{(2\xi)}\right)$.
The parameter $\xi$ controls the width of the Gaussian kernel. Choosing $k(x,y)$ to be circularly symmetric is for convenience; a Gaussian kernel with different width parameters for the $x$ and $y$ directions can still be used to construct $k(x,y)$ that is doubly symmetric for diagonalization via the $\mathtt{DCT}$ \cite{HansenNagyOLeary}. In regards to the value of $\xi$, values $\xi = 4$, $16$, and $36$ correspond to blurring that is referred to as ``mild," ``medium," and ``severe", respectively,  \cite{IRTools}. The corresponding $k(x,y)$ were discretely convolved with each image as a means of blurring. SNR values of $10$, $25$ and $40$ were used to construct mean zero  independent Gaussian noise vectors  that were added to the blurred images to create the data. For one choice of the penalty matrix $\bfL$, we used the appropriately structured version of the discrete negative Laplacian operator, which is an approximation of the continuous Laplacian operator \cite{DebnathMikusinski2005,LeVeque2007}, and which we denote by $\bfL=\bfL_2$. For the second penalty matrix we used $\bfL=\bfeye$. The structure of $\bfA$ and $\bfL$ allows for simultaneous diagonalization using the $\mathtt{DCT}$ for numerical efficiency (see  \Cref{sec:Decomp}). 

\inserted{The learning methods were evaluated for both the scalar and spectral windowing cases using training data sets of sizes $R=1$ to $8$.   The learned parameters in each case were  then used to construct regularized solutions for data from two independent validation sets. For the windowed regularization we considered both non-overlapping linear/logarithmic windows and overlapping linear/logarithmic cosine windows. The decision to use linear spacing for $\bfL = \bfeye$ and logarithmic spacing for $\bfL = \bfL_2$ is supported by how the ordered spectral components decay \cite{Byrne}. As in \cite{ChEaOl:11}, the windowed $\mathtt{GCV}$ function \cref{thm:windowgcvfunc} was replaced by the $P$ independent approximate $\mathtt{GCV}$ functions \cref{eq:GCV_Big Decoupled} for simplicity when considering non-overlapping windows. Parameters were also obtained for the separable $\mathtt{UPRE}$ method given by \cref{eq:UPREwinsepR}. For the spectral windowing with overlapping windows, the minimizations were initialized using the parameters obtained by the non-overlapping methods. Overall, in terms of the choice to initialize the parameters for the overlapping windows with parameters obtained from the separable case, we note that the windowed $\mathtt{UPRE}$ and $\mathtt{GCV}$ methods corresponding to overlapping windows performed better when the minimizations were initialized using the parameter obtained by the non-overlapping methods. Results without this initialization are not given.}

\inserted{Considering first the scalar parameter MD case, the resulting parameters appear to stabilize as the number of data sets is increased. \Cref{fig:Parameter trends} demonstrates this effect and shows that the amount of stabilization appears to be connected to the homogeneity of the training set. Sets constructed from \cref{fig:MESSENGER True} are homogeneous in the sense that they all contain images of the surface of Mercury. In contrast, \cref{fig:Validation Set 2} consisted of entirely distinct images. By changing which sets are used for training or validation influences the resulting parameters as $R$ increases. The corresponding relative errors of the regularized solutions are shown in \cref{fig:MD Errors}. While the box plots in \cref{Errors} and \cref{ErrorsSwitched} simply look as if they are the same but reordered, the box plots appear similar because the resulting parameters are approximately the same ($\regparam \approx 0.2$). These experiments suggest that it is sufficient to use only a small number of images, relative to the total available, to obtain meaningful results. The use of \cref{fig:Validation Set 2} as a training set was only considered to produce \cref{fig:Parameter trends} and \cref{fig:MD Errors}; its use as a validation set is retained through the remainder of the results.}

\begin{figure}[htbp]
\centering
\subfigure[Parameters $\regparam$ against $R$ (Set 1) \label{Trend}]{\includegraphics[width=0.49\textwidth]{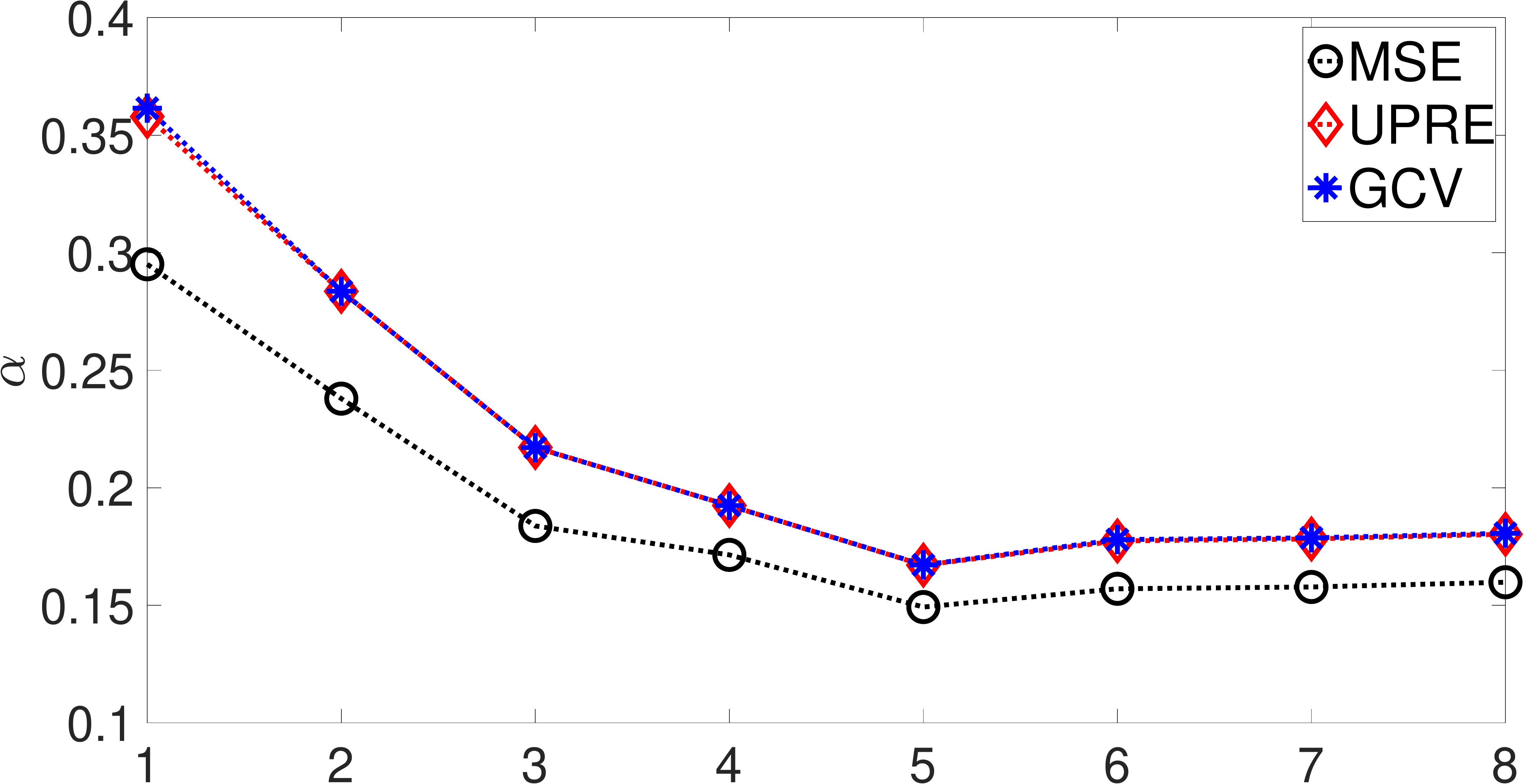}}\hspace{.1cm}
\subfigure[Parameters $\regparam$ against $R$ (Set 3) \label{TrendSwitched}]{\includegraphics[width=0.49\textwidth]{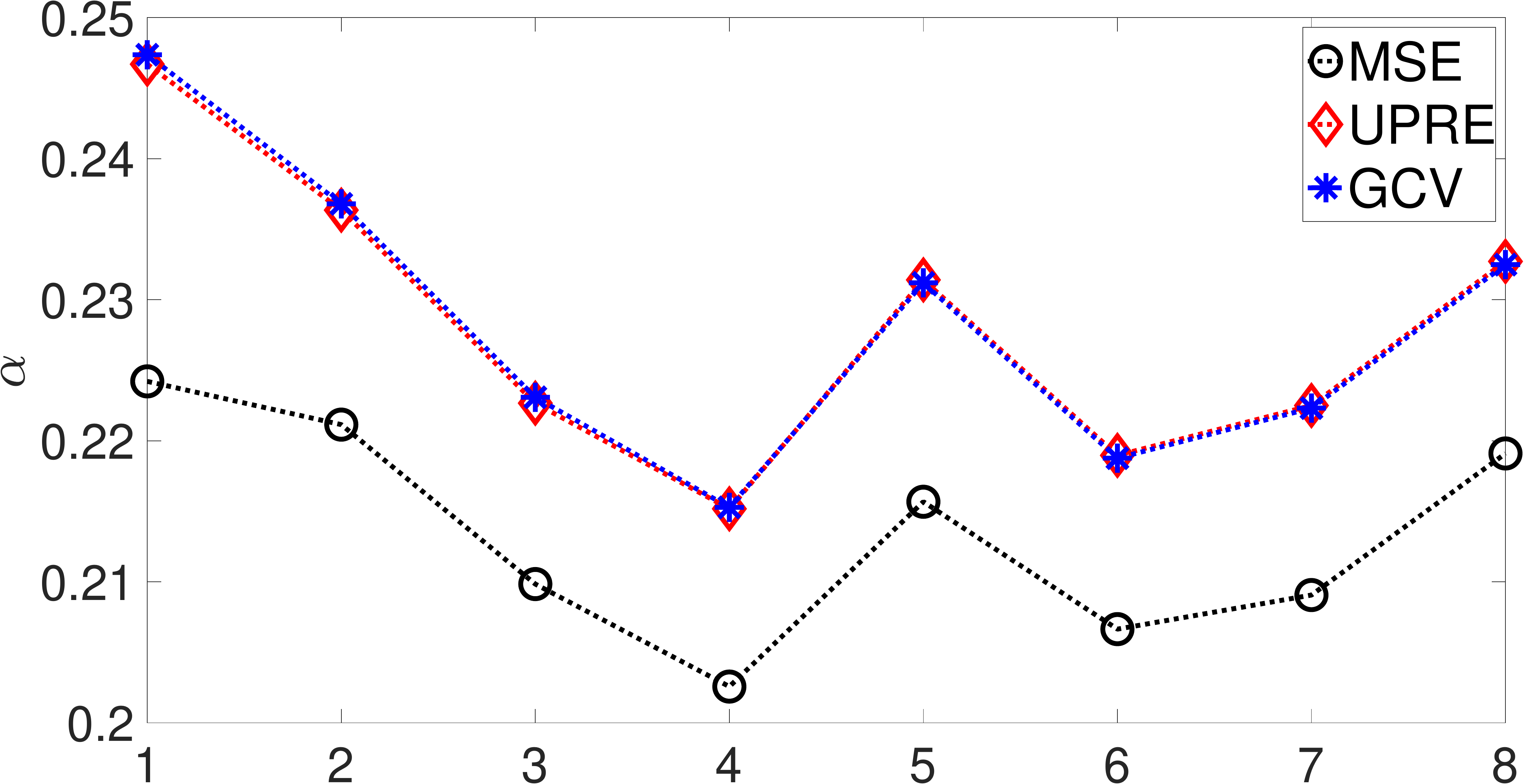}}
\caption{\inserted{\Cref{Trend} illustrates the change in scalar $\regparam$ as the number of data sets increases, here with Set 1 as the training set and in  \cref{TrendSwitched} with  Set 3 (see \cref{fig:MD Errors}). \Cref{Trend} is an example of how scalar regularization parameters can stabilize as the number of data sets in the MD methods increases. In contrast,  \cref{TrendSwitched} shows less stabilization with increasing $R$ when the training set is changed. For both plots, $\xi = 16$, $\bfL = \bfL_2$, and an SNR of 25 was used.}}
\label{fig:Parameter trends}
\end{figure}

\begin{figure}[htbp]
\centering
\subfigure[Relative errors (Set 1) \label{Errors}]{\includegraphics[width=0.49\textwidth]{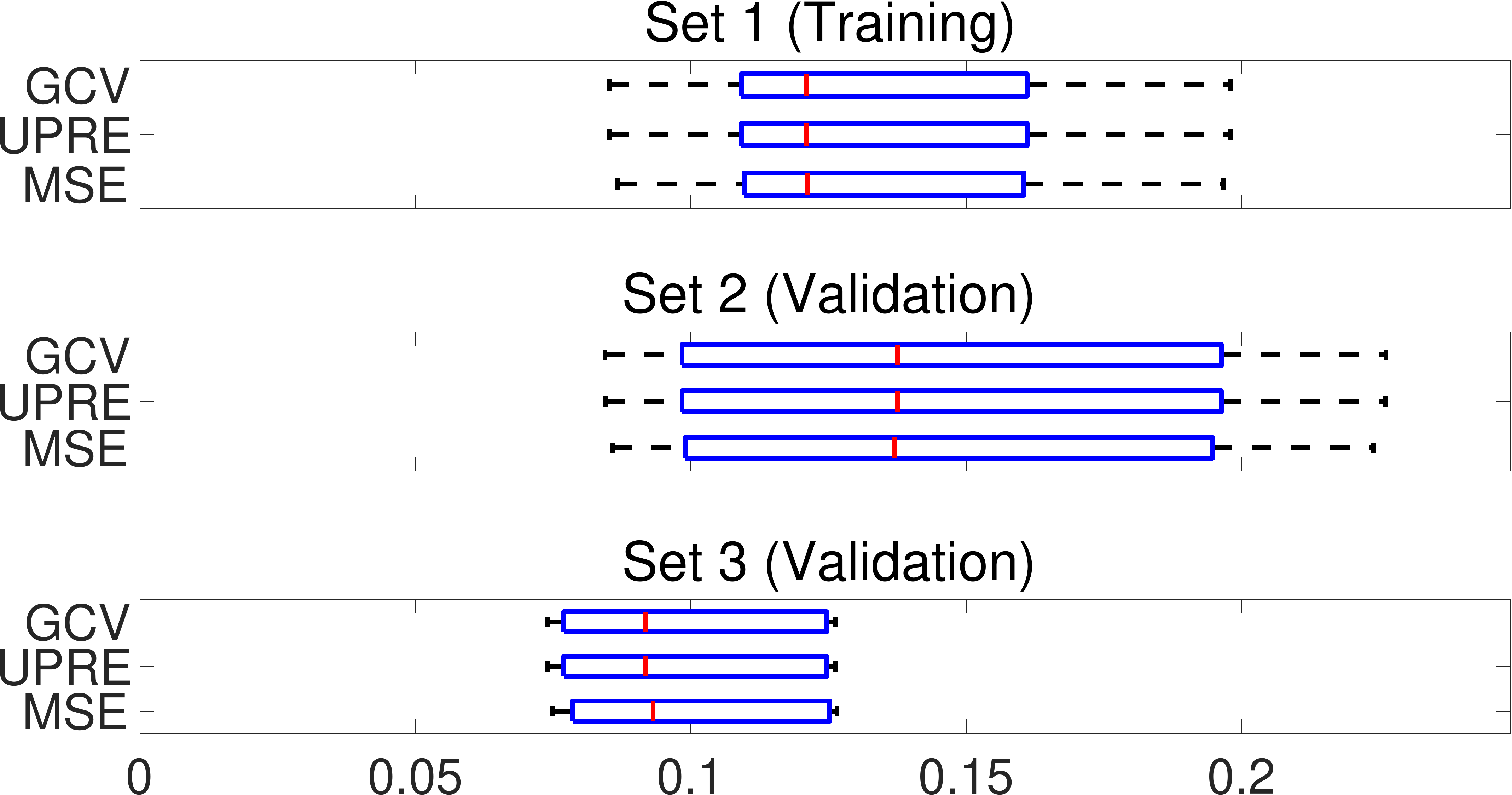}}\hspace{.1cm}
\subfigure[Relative errors (Set 3) \label{ErrorsSwitched}]{\includegraphics[width=0.49\textwidth]{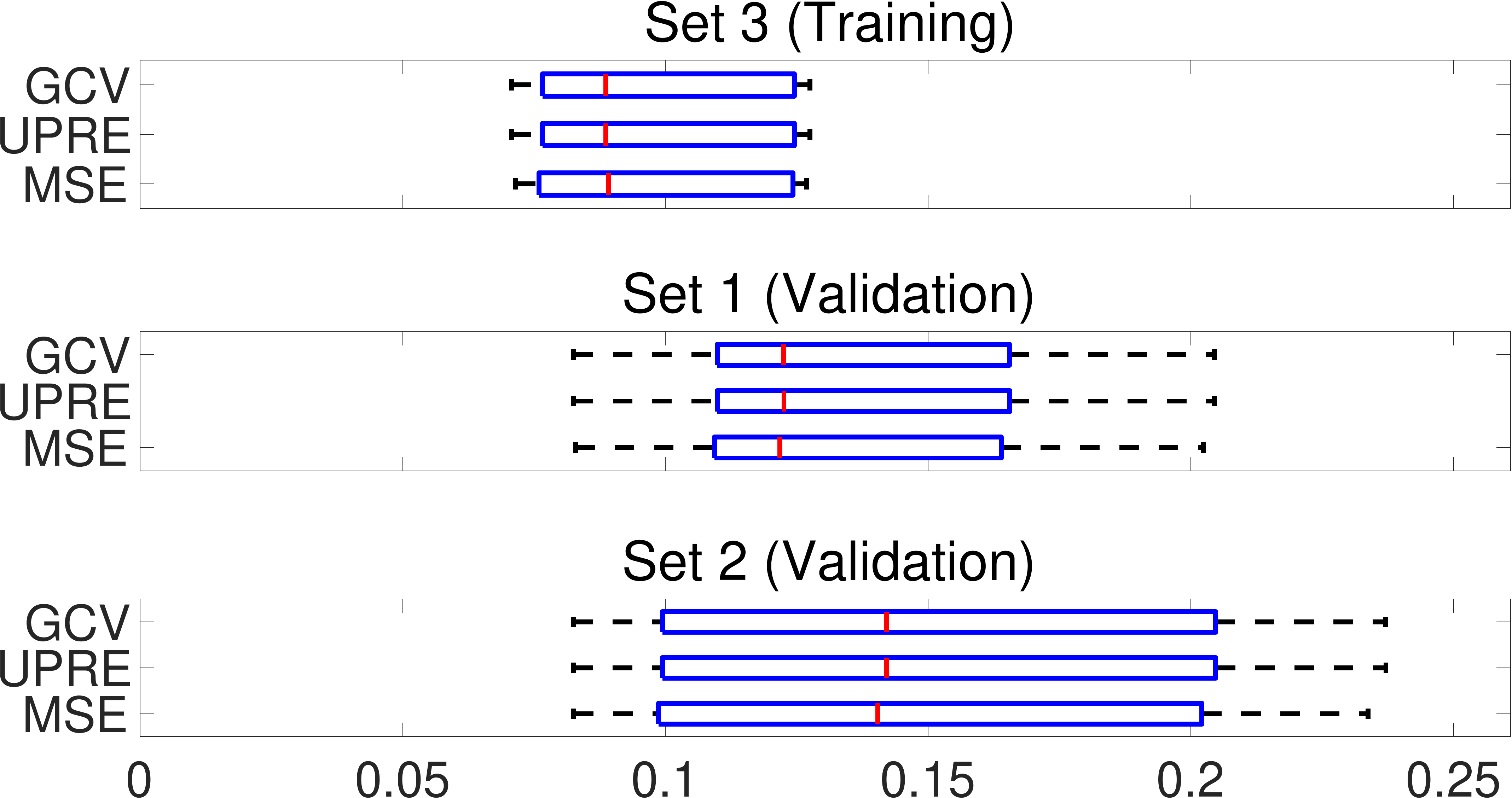}}
\caption{\inserted{Relative errors of regularized solutions obtained for scalar $\regparam$ from each MD method with $R = 5$ data sets. Here Set 1 and Set 2 were constructed from \cref{fig:MESSENGER True}, while Set 3 was constructed from \cref{fig:Validation Set 2}. In \cref{Errors}, Set 1 served as the training set and the resulting parameters were used to construct solutions for the data from Sets 2 and 3. \Cref{ErrorsSwitched} shows results where training was done using Set 3 instead and Sets 1 and 2 served as validation sets. For both plots, $\xi = 16$, $\bfL = \bfL_2$, and an SNR of 25 was used.}}
\label{fig:MD Errors}
\end{figure}

\inserted{In regards to the spectral windowing, typically two windows were sufficient (corresponding to the use of just two parameters in the windowed estimators) to obtain meaningful solutions. The observed benefit of using greater than two windows was minor, an example of which is shown for the windowed $\mathtt{UPRE}$ method in \cref{fig:UPRE Errors}. Another advantage of using two windows is that there is a greater computational cost of finding more parameters than is necessary for meaningful regularized solutions; this is especially true for overlapping windows where decoupling is not an option. Extending the number of windows also has the effect of reducing the influence of one or more parameters. For example, \cref{fig:UPRE Parameters} shows that one of the three parameters obtained from the windowed $\mathtt{UPRE}$ method with three windows is more variable and larger in magnitude than that other two parameters.}

\begin{figure}[htbp]
\centering
\subfigure[Relative errors \label{fig:UPRE Errors}]{\includegraphics[width=0.49\textwidth]{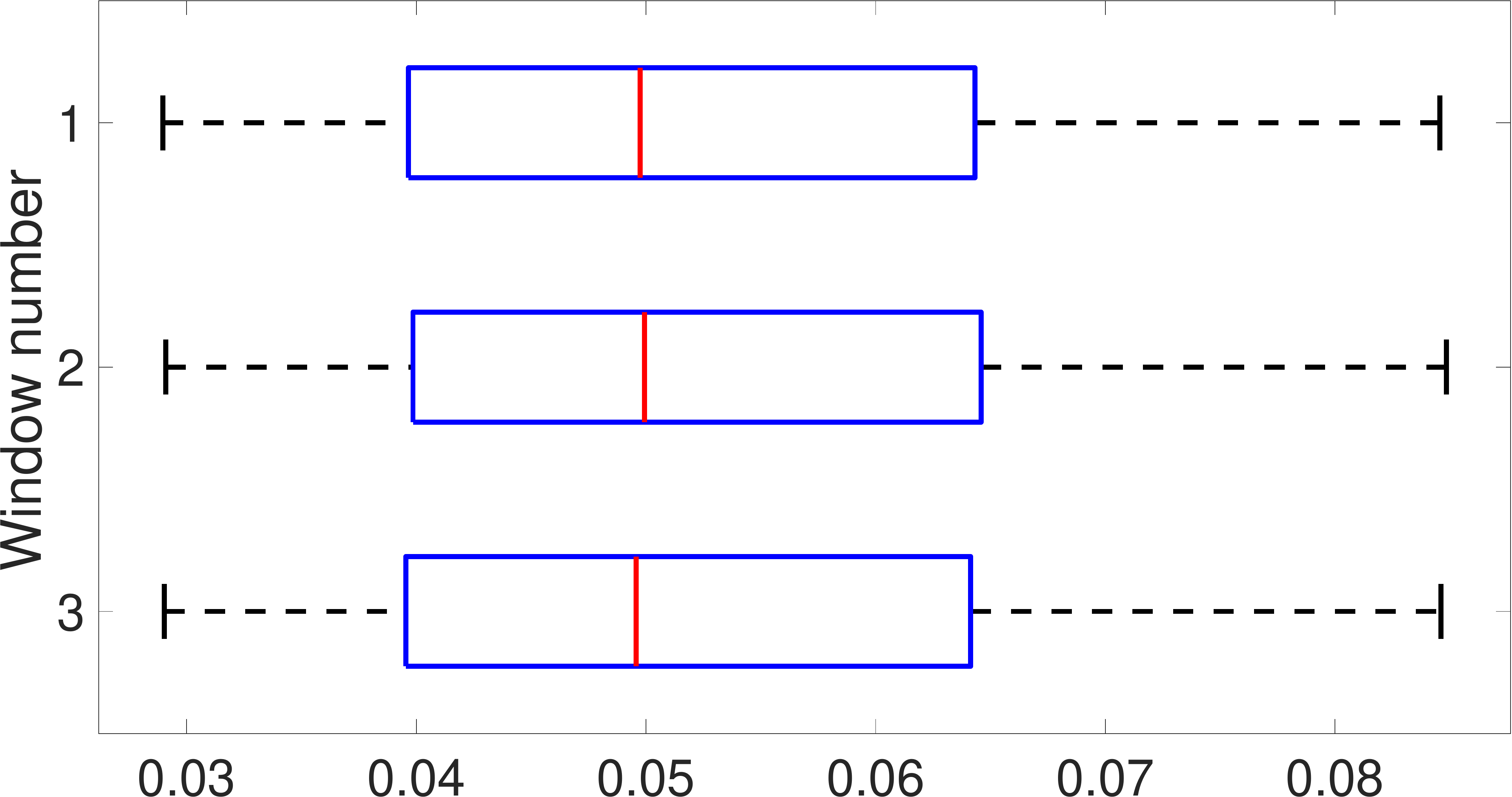}}
\hspace{.1cm}
\subfigure[Parameters for $P = 1:3$ \label{fig:UPRE Parameters}]{\includegraphics[width=0.49\textwidth]{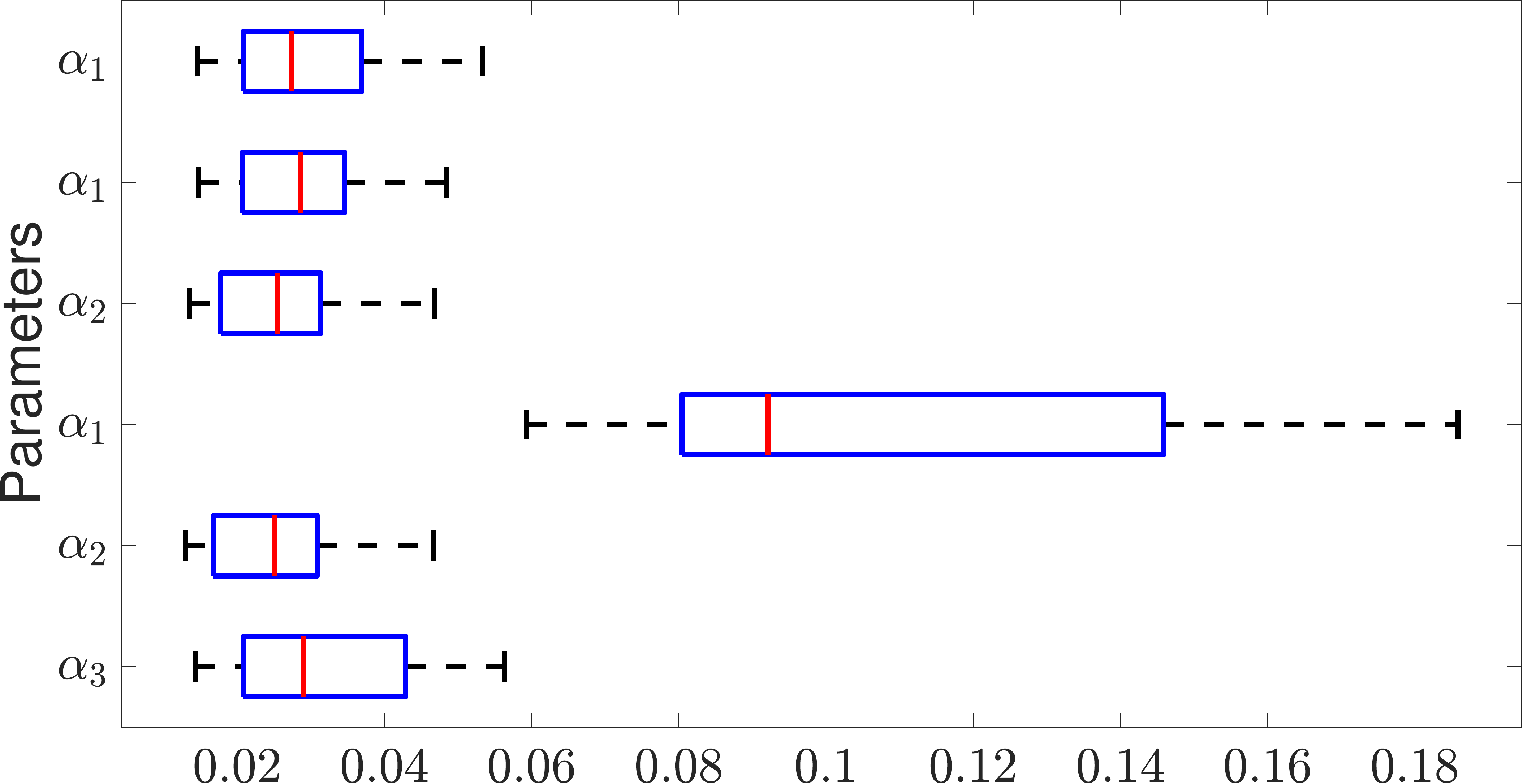}}
\caption{\inserted{Parameters and corresponding relative errors from the $\mathtt{UPRE}$ method as the number of windows is increased from one to three. \Cref{fig:UPRE Errors} shows that there is little benefit in using an increasing number of windows. \Cref{fig:UPRE Parameters} shows that past two windows, the new regularization parameters are more variable. For both plots, logarithmically spaced windows were used with $\xi = 4$, $\bfL = \bfL_2$, and an SNR of 40.}}
\label{fig:UPRE Parameters and Errors}
\end{figure}


The results presented in \cite{ChEaOl:11} \inserted{also} suggested that there is little to be gained when using more than two windows, even when using the learning approach, method $\mathtt{MSE}$, to find the parameters. 
On the other hand, the presented framework is valid for more windows, should there be situations in which the use of two windows seems insufficient based on numerical experiments. 
It should be noted also, that when using the MD windowed $\mathtt{MDP}$ method, there is an additional tuning safety parameter, which is required and makes the presentation of results for the $\mathtt{MDP}$ method much less interesting, see \cite{Byrne}.

\begin{figure}[htbp]
\centering
\includegraphics[width=0.9\textwidth]{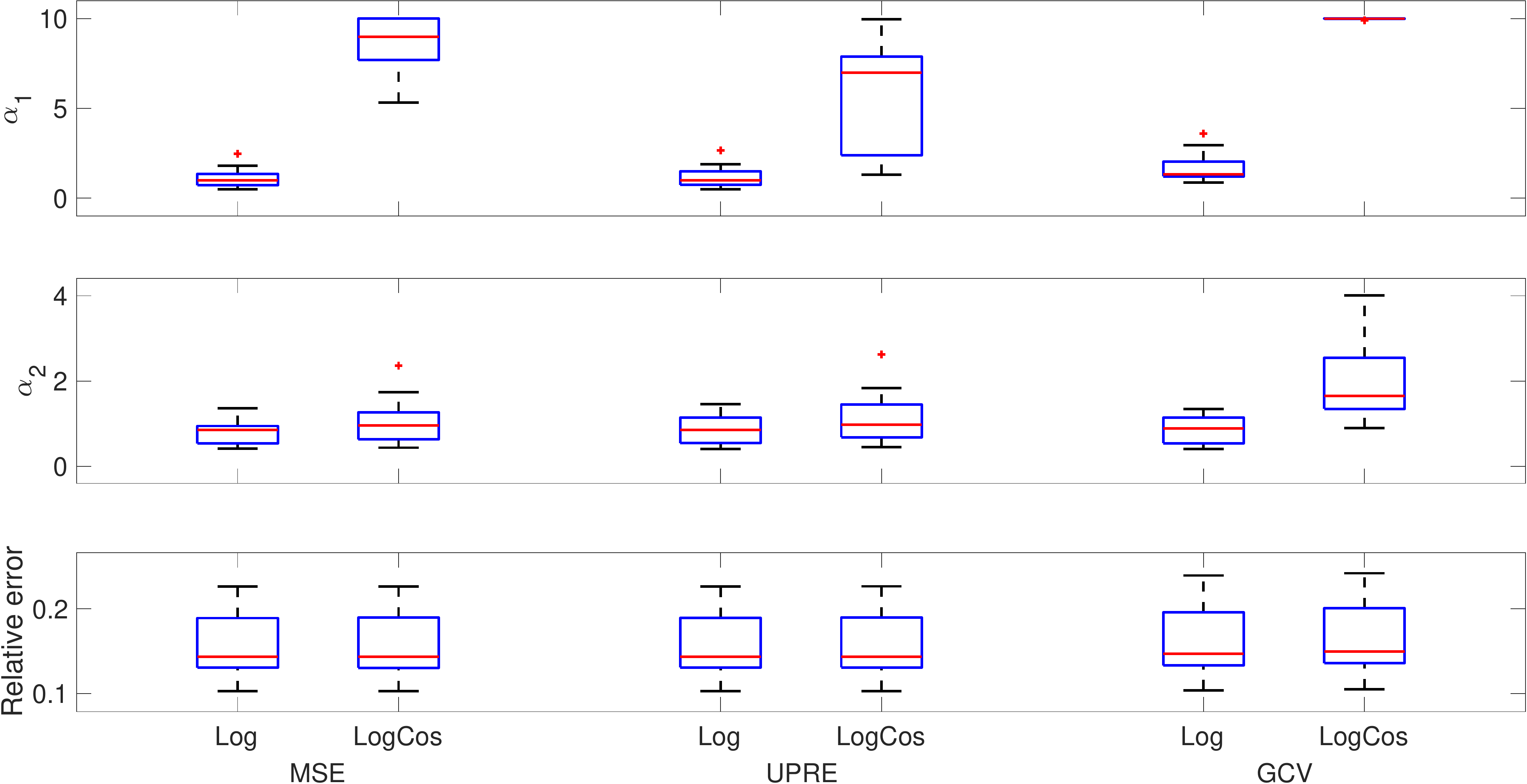}
\caption{\inserted{Parameters and corresponding relative errors obtained from using logarithmic vs logarithmic-cosine windows with the $\mathtt{MSE}$, $\mathtt{UPRE}$, and $\mathtt{GCV}$ methods. In the case of the logarithmic (non-overlapping) windows, the independent versions of the $\mathtt{UPRE}$ and $\mathtt{GCV}$ functions were used, \cref{eq:UPREwinsepR,eq:GCV_Big Decoupled}, respectively. For both window versions, $\xi = 4$, $\bfL = \bfL_2$, and an SNR of 10 was used.}}
\label{fig:Boxplots_LogvsLogCos_v4_SNR10_L}
\end{figure} 

The use of overlapping or non-overlapping windows influences the degree of interdependence between the two parameters.   \Cref{fig:Boxplots_LogvsLogCos_v4_SNR10_L} presents the results of using overlapping and non-overlapping logarithmic windows with $\bfL = \bfL_2$. When using non-overlapping windows, the ranges of both parameters are smaller than those for non-overlapping windows. For overlapping windows, the behavior of $\regparam_1$ exhibited in \cref{fig:Boxplots_LogvsLogCos_v4_SNR10_L} shows the parameters grouping near $10$. The grouping behavior can be explained by the choice of an upper bound during the minimization process; in the case of \cref{fig:Boxplots_LogvsLogCos_v4_SNR10_L}, the upper bound was chosen near $10$. The calculated gradients of the $F_{\substack{\text{$\mathtt{MSE}$} \\ \mathtt{win}}}(\regparamVec)$, $F_{\substack{\text{$\mathtt{UPRE}$} \\ \mathtt{win}}}(\regparamVec)$, and $F_{\substack{\text{$\mathtt{GCV}$} \\ \mathtt{win}}}(\regparamVec)$ are too small to resolve a minimum in the direction of $\regparam_2$ and thus the minimization process determines the minimizers near the specified boundary. However, using overlapping windows also increased the magnitude of $\regparam_2$, most significantly in the case of the $\mathtt{GCV}$ method.

In regards to the MD windowed methods, which select $P$ parameters using $R$ data sets, the parameters converge as $R$ increases. \Cref{tab:All Relative Errors v36_SNR10_Identity} details the mean percent relative errors of solutions obtain using parameters from each MD windowed method, where one and two (both overlapping and non-overlapping) windows were used. Even for the limited number of training sets ($2$ through $8$), the errors decrease as $R$ increases. \inserted{For most numerical configurations tested, the use of overlapping vs non-overlapping windows provides minor benefit with regards to the relative errors of the regularized solutions.}


\begin{table}[htbp]
  \begin{center}
    \caption{\inserted{Averaged percent relative errors of the MD windowed regularized solutions for $\xi = 36$ and an SNR of $10$ with one window, two linearly spaced windows and two linearly spaced cosine windows with the identity penalty matrix. The result with least error for given $R$, method, and validation set is highlighted in bold face.}}
    \label{tab:All Relative Errors v36_SNR10_Identity}
    \resizebox{\textwidth}{!}{%
    \begin{tabular}{|c|c|c|c|c|c|c|c|c|c|c|}
    \hline
    \multirow{2}{*}{$R$} & \multirow{2}{*}{$\mathtt{Win}$} & \multicolumn{3}{|c|}{Training} & \multicolumn{3}{|c|}{Validation 1} & \multicolumn{3}{|c|}{Validation 2} \\
    \cline{3-11}
     & & $\mathtt{MSE}$ & $\mathtt{UPRE}$ & $\mathtt{GCV}$ & $\mathtt{MSE}$ & $\mathtt{UPRE}$ & $\mathtt{GCV}$ & $\mathtt{MSE}$ & $\mathtt{UPRE}$ & $\mathtt{GCV}$ \\
    \hline
    \multirow{3}{*}{$2$} & $\mathtt{None}$ & $\mathbf{21.32}$ & 25.83 & 25.87 & $\mathbf{23.76}$ & 27.58 & 27.62 & $\mathbf{17.57}$ & 23.27 & 23.32 \\
    \cline{2-11}
    & $\mathtt{Lin}$ & 19.64 & $\mathbf{19.54}$ & $\mathbf{19.54}$ & 22.85 & $\mathbf{22.67}$ & $\mathbf{22.67}$ & 14.94 & $\mathbf{14.88}$ & 14.89 \\
    \cline{2-11}
    & $\mathtt{LinCos}$ & $\mathbf{19.54}$ & 19.93 & 19.94 & $\mathbf{22.83}$ & 23.38 & 23.39 & $\mathbf{14.78}$ & 15.10 & 15.11 \\
    \hline
    \multirow{3}{*}{$4$} & $\mathtt{None}$ & $\mathbf{21.29}$ & 26.09 & 26.12 & $\mathbf{23.71}$ & 27.82 & 27.86 & $\mathbf{17.58}$ & 23.57 & 23.62 \\
    \cline{2-11}
    & $\mathtt{Lin}$ & $\mathbf{19.44}$ & 19.53 & 19.55 & 22.39 & $\mathbf{22.36}$ & 22.37 & $\mathbf{14.93}$ & 15.16 & 15.17 \\
    \cline{2-11}
    & $\mathtt{LinCos}$ & $\mathbf{19.32}$ & 19.33 & 19.94 & 22.32 & $\mathbf{22.29}$ & 23.39 & $\mathbf{14.75}$ & 14.79 & 15.11 \\
    \hline
    \multirow{3}{*}{$6$} & $\mathtt{None}$ & $\mathbf{21.29}$ & 26.03 & 26.07 & $\mathbf{23.71}$ & 27.77 & 27.81 & $\mathbf{17.58}$ & 23.51 & 23.56 \\
    \cline{2-11}
    & $\mathtt{Lin}$ & $\mathbf{19.44}$ & 19.49 & 19.50 & 22.41 & $\mathbf{22.36}$ & 22.37 & $\mathbf{14.91}$ & 15.05 & 15.07 \\
    \cline{2-11}
    & $\mathtt{LinCos}$ & $\mathbf{19.32}$ & $\mathbf{19.32}$ & 19.94 & $\mathbf{22.34}$ & $\mathbf{22.34}$ & 23.39 & 14.73 & $\mathbf{14.72}$ & 15.11 \\
    \hline
    \multirow{3}{*}{$8$} & $\mathtt{None}$ & $\mathbf{21.29}$ & 26.02 & 26.06 & $\mathbf{23.71}$ & 27.76 & 27.80 & $\mathbf{17.57}$ & 23.50 & 23.55 \\
    \cline{2-11}
    & $\mathtt{Lin}$ & $\mathbf{19.44}$ & 19.48 & 19.49 & 22.41 & $\mathbf{22.36}$ & 22.37 & $\mathbf{14.90}$ & 15.04 & 15.06 \\
    \cline{2-11}
    & $\mathtt{LinCos}$ & $\mathbf{19.32}$ & $\mathbf{19.32}$ & 19.94 & $\mathbf{22.34}$ & $\mathbf{22.34}$ & 23.39 & $\mathbf{14.72}$ & $\mathbf{14.72}$ & 15.11 \\
    \hline
    \end{tabular}}
  \end{center}
\end{table}

It is interesting to note that the average relative errors of solutions obtained for parameters applied to the second validation set (\cref{fig:Validation Set 2}) were less than those of either the training or first validation set. 
\inserted{The superior (reduced) errors calculated for the second validation set are consistent throughout most numerical configurations. Additionally, the relative errors for the second validation set show  greater variability than those for either the training,  or first validation, set. Furthermore,  the relative errors are indeed least in each case when training is performed using known data, namely with the \texttt{MSE}, but the results with both \texttt{UPRE} and \texttt{GCV} learning methods are not significantly larger  when using windowed regularization. This demonstrates that windowed regularization parameters can be learned from training data without knowledge of the true solutions. The results obtained using \texttt{UPRE} are in most cases slightly improved as compared to those using  \texttt{GCV}, and hence \texttt{UPRE} would be preferred if information about the noise in the data is available. Finally, to illustrate the performance of the approach, 
\cref{fig:Sample_Solutions} presents two examples of images from the second validation set that have differing relative errors.}
\begin{figure}[htbp]
\centering
\includegraphics[width=0.9\textwidth]{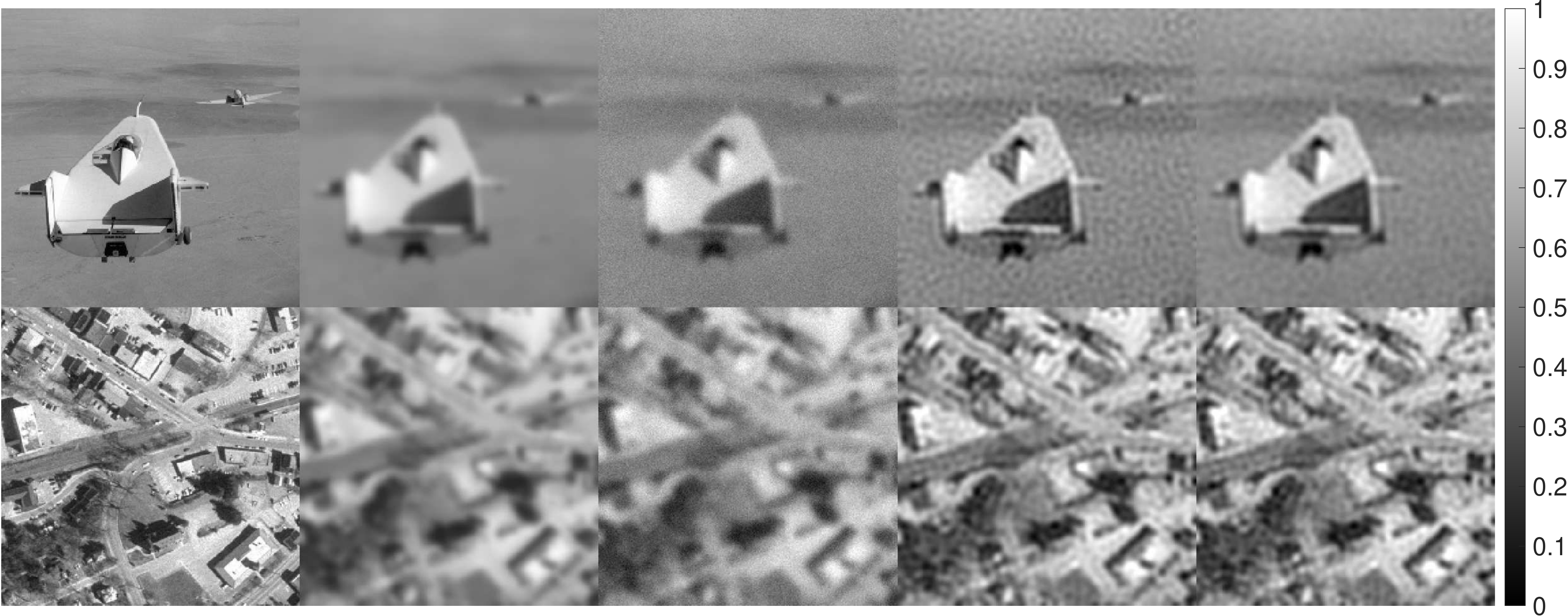}
\caption{\inserted{Two samples from the second validation set, with $\xi = 36$, an SNR of $25$, two log cosine windows and the Laplacian penalty matrix. From left to right for each sample are the true solution, the blurred image, the blurred image after noise was added, the regularized solution obtained using the MD windowed $\mathtt{UPRE}$ method with $R = 8$ (the entire training set), and the regularized solutions using parameters that are optimal for the individual image. The MD windowed $\mathtt{UPRE}$ solutions have relative errors of $8.55\%$ and $14.81\%$ for the top and bottom samples, respectively, while the optimal solutions have relative errors of $8.04\%$ and $14.80\%$.}}
\label{fig:Sample_Solutions}
\end{figure}

\section{Conclusions}\label{sec:Conclusion}
We have shown that the $\mathtt{UPRE}$ and $\mathtt{GCV}$ methods can be extended to accommodate regularization parameter estimation using multiple data sets and for both single and windowed regularization parameters, for generalized Tikhonov regularization. The $\mathtt{UPRE}$ is a  representative estimator that assumes the knowledge of the variance of mean zero Gaussian noise in the data, while no additional assumptions are required for the $\mathtt{GCV}$ estimator.  The most general forms of functions associated with these methods are given in \Cref{thm:Main UPRE Result} and  \Cref{thm:WindowedGCV}. While the corresponding function for the MD windowed $\mathtt{UPRE}$  can be written as an average of the individual functions associated with each data set, this is not possible for the MD windowed $\mathtt{GCV}$. Moreover, the $\mathtt{GCV}$ estimator for windowed regularization parameters when derived from first principles is more complex, and unlike the $\mathtt{UPRE}$ case, does not yield a separable form when non-overlapping windows are applied. Still, neither of these  MD windowed methods require knowledge of true solutions unlike the learning approach defined by \cref{eq:Learning function}. The presented numerical experiments for  $2$D signal restoration  demonstrate that the MD windowed methods can perform competitively with the learning approach that requires knowledge of true signals for training from data. Further, it is also demonstrated that the parameters obtained from a specific training set of validation images can also be used for a set of different testing images, provided that the general noise characteristics are the same. \edited{Varying the noise characteristics of the training and/or validation sets could be an interesting approach for future investigation.}

\edited{We note that the approach discussed here will extend immediately for any estimator which relies only on an approximation for the regularized residual and the trace of the data resolution matrix, as is the case for  $\mathtt{UPRE}$, provided that the derivation from first principles still leads to an equivalent formulation for the function that should be minimized. The general idea can be modified to address estimators requiring other terms, such as an augmented regularized residual used for the $\chi^2$ estimator described in \cite{mead:08,mere:09}, and as already given in \cite{Byrne} for the $\mathtt{MDP}$ method. Although the implementation is presented for the case in which there is a known mutual decomposition of the model and penalty matrices ($\bfA$ and $\bfL$) the approach can be extended for any iterative method which yields suitable estimates, e.g. \cite{ChNaOl:08,RVA:15,saeed6,saeed7}. In particular, the assumption of a mutual diagonalization is only a simplifying assumption that is useful computationally for the $\mathtt{UPRE}$ method. Consequently, the approach here can also be used with other Kronecker product representations for the model and penalty matrices, or for Krylov methods that can be used to provide a spectral representation of the two operators. Such extensions would be a topic for future research, as would a true multi penalty regularization using multiple data sets for learning parameters to meet conditions for $\mathtt{UPRE}$. We conclude that  the results are particularly helpful in demonstrating that $\mathtt{UPRE}$ offers a major advantage, as compared to $\mathtt{GCV}$;  the derivation from first principles yields a standard $\mathtt{UPRE}$ function, whereas the $\mathtt{GCV}$ does not.  This suggests that  $\mathtt{UPRE}$ will be of more general use in conjunction with iterative solvers replacing the direct solves, which is not the case for the $\mathtt{GCV}$ formulation. Moreover, $\mathtt{UPRE}$ is immediately separable for non-overlapping windows but  $\mathtt{GCV}$ is not.}

\appendix

\bibliographystyle{plain}
\bibliography{ParameterEstimation}
\section{The Unbiased Predictive Risk Estimator for Windowed Tikhonov}\label{app:UPRE}
We  briefly describe the derivation of the $\mathtt{UPRE}$ in the context of learning regularization parameters from multiple data sets and using spectral windowing. This mimics the derivation in  \cite[p.~98]{Vogel:2002} but in all cases assumes that the matrices and vectors are for multiple data sets and that regularization parameters are defined with respect to spectral windows.  For ease of notation we drop the  $\widetilde{\bfA}$ notation, and we drop the dependence on $\regparamVec$ in all terms. Also we assume that the noise in the measured data   is independent so that the covariance matrix for the noise, $\bfSigma$, is diagonal, where $\noiseVec \sim \mathcal{N}(0, \bfSigma)$. We now proceed to obtain the proof of \Cref{thm:Main UPRE Result} from first principles. 
 \begin{proof} 
 Using $\xReg=\AWin^\sharp \dVec$  we have  
 \begin{align*}
     \rVec(\regparamVec)&= (\bfA\AWin^\sharp-\bfeye_M)\dVec  = (\bfA\AWin^\sharp-\bfeye_M)\bVec + (\bfA\AWin^\sharp-\bfeye_M)\noiseVec \text{ and }\\
     \pVec(\regparamVec) &= (\bfA \AWin^\sharp -\bfeye_M) \bVec + \bfA \AWin^\sharp\noiseVec.
 \end{align*}
Here  $\bfeye_M$ is defined consistently for the size of the problem, $M=\sum_{r=1}^Rm_r$. In both cases we have a linear combination  between $\fVec=(\bfA\AWin^\sharp-\bfeye_M)\bVec$ which is deterministic and a noise term $\bfB\noiseVec$, where $\bfB=\bfA \AWin^\sharp$ or $\bfB=(\bfA \AWin^\sharp-\bfeye_M)$.  Now by the Trace Lemma \cite[Lemma 7.2]{Vogel:2002} we have 
 \begin{align}\label{eq:trace}
     E(\|\fVec+ \bfB\noiseVec\|_2^2)= E(\|\fVec\|_2^2)+ \trace(\trans{\bfB}\bfB\bfSigma).
 \end{align}
  Here $E(a)$ denotes the expectation of $a$. Thus, applying \cref{eq:trace} twice  we have 
 \begin{align*}
    E(\|\rVec(\regparamVec)\|_2^2)&= E(\|\fVec\|_2^2 ) +\trace(\trans{(\bfA\AWin^\sharp)}\bfA\AWin^\sharp\bfSigma) \\
    &\quad \quad + \trace(\bfSigma) -  \trace(\bfA\AWin^\sharp\bfSigma) -\trace(\trans{(\bfA\AWin^\sharp)}\bfSigma) \text{ and }
    \\ 
 E(\|\pVec(\regparamVec)\|_2^2)& =E( \|\fVec\|_2^2) + \trace(\trans{(\bfA\AWin^\sharp)}\bfA\AWin^\sharp\bfSigma).
 \end{align*}
 Immediately, approximating $E(\|\rVec(\regparamVec)\|_2^2)$ by $\|\rVec(\regparamVec)\|_2^2$ as in the single parameter derivation, and assuming that  $\bfA\AWin^\sharp$ is symmetric, yields the estimator 
\begin{align}\label{eq:UPREwin}
     E(\frac{1}{M}\|\pVec(\regparamVec)\|_2^2)
     &\approx \frac{1}{M}\left(\|\rVec(\regparamVec)\|_2^2 -\trace(\bfSigma)+  2\,\trace(\bfA\AWin^\sharp\bfSigma)\right).
 \end{align}
For  multiple data sets the system matrices are block diagonal and the systems  decouple. Hence  $\|\rVec(\regparamVec)\|_2^2 = \sum_{r=1}^R \|\rVec^{(r)}(\regparamVec)\|_2^2$ and $$\trace(\bfA\AWin^\sharp\bfSigma)= \sum_{r=1}^R \trace(\bfA^{(r)}(\AWin^{(r)})^\sharp\bfSigma_r).$$
 Therefore, 
  \begin{align*}
     E(\frac{1}{M}\|\pVec(\regparamVec)\|_2^2) &\approx 
     \frac{1}{M}\sum_{r=1}^R  \left(\|\rVec^{(r)}(\regparamVec)\|_2^2 -\trace(\bfSigma_r)+2\,  \trace(\bfA^{(r)}(\AWin^{(r)})^\sharp\bfSigma_r) \right) \\
     &=\frac{1}{M}\sum_{r=1}^R m_r\U^{(r)}(\regparamVec), 
     \end{align*}
where the last equality results by applying \cref{eq:UPREwin} for each system $r$. 
As a consequence we have the proof of the estimator given in \Cref{thm:Main UPRE Result} under the assumption that $\bfA\AWin^\sharp$ is symmetric
\begin{align*}
 E(\frac{1}{M}\|\pVec(\regparamVec)\|_2^2) &\approx 
     \frac{1}{M}\sum_{r=1}^R  \left(\|\rVec^{(r)}(\regparamVec)\|_2^2 -\trace(\bfSigma_r) 
     +  2\trace(\bfA^{(r)}(\AWin^{(r)})^\sharp\bfSigma_r)\right).  
\end{align*}
\end{proof}
When we have the GSVD for each system it is immediate, with or without  windowing, that  
\begin{align}\label{eq:AwinAwnsharp}
\bfA\AWin^\sharp= U  \sum_{p=1}^P\left(W^{(p)})^{1/2} \Delta\inv{(\trans{\Delta}\Delta+\alpha_p^2 \trans{\Lambda}\Lambda)} \trans{\Delta} (W^{(p)})^{1/2}\right) \trans{U},\end{align}  is symmetric for each block and/or window.

\section{Generalized Cross Validation for Windowed Tikhonov}\label{app:GCV}
We now consider the direct derivation from first principles for the generalized Tikhonov regularized problem.  Following the approach in \cite{ChEaOl:11}, and using \cref{eq:GSVD Normal Eq Sol}, we  consider the solution using the spectral domain coefficients
\begin{align}\label{eq:trans}
  \yVec= \trans{\bfX}\xVec(\regparam) &=   \bfPhi(\regparam)\pinv{\bfDelta}\hat{\dVec},
  \end{align}
  which corresponds to the solution of the normal equations
  \begin{align}
(\trans{\bfDelta}{\bfDelta}+ \regparam^2\trans{\bfLambda} {\bfLambda})\yVec 
  &= \trans{\bfDelta}\hat{\dVec},
 \end{align}
for the regularized problem 
 \begin{align}\label{eq:GCVnormalytildefull}
    \|\bfDelta \yVec -\hat{\dVec}\|_2^2 + \regparam^2\|\bfLambda{\yVec}\|_2^2.
\end{align} 
As in \cite{ChEaOl:11,GoHeWa}, we introduce $\bfC$ as the unitary matrix which diagonalizes the circulants and consider the new system $\bfG {\yVec}  \approx \bfC \hat{\dVec}$, where $\bfG=\bfC{\bm{\Delta}}$  is of size $\mA\times n$. For this new system $\yVec$ solves 
\begin{align*}
    (\trans{\bfG}{\bfG} + \regparam^2 \trans{\bfLambda} {\bfLambda}) \yVec = \trans{\bfG}\bfC{\hat{\dVec}},
\end{align*}
and the corresponding windowed solution is given by 
\begin{align}\label{eq:winTransform}
 \yWin = \sum_{p=1}^P \inv{(\trans{\bfG}{\bfG} + \regparam_p^2 \trans{\bfLambda} {\bfLambda})} \bfW^{(p)}\trans{\bfG} \bfC{\hat{\dVec}}.
\end{align}
Following \cite{ChEaOl:11} but consistent with our notation introduced in \cref{eq:normalequationsolution} applied for the solution in the spectral domain  
we introduce the data resolution matrices 
\begin{align*}
    {\bfG}(\regparam_p)&=\bfG\inv{(\trans{\bfG}{\bfG} + \regparam_p^2 \trans{\bfLambda} {\bfLambda})} \trans{\bfG}= \bfG \bfG^\sharp(\regparam_p) \text{ and  } \\
    \bfG_{\mathrm{win}^{(p)}}(\regparam_p)&= \bfG \inv{(\trans{\bfG}{\bfG} + \regparam_p^2 \trans{\bfLambda} {\bfLambda})}\bfW^{(p)} \trans{\bfG}=\bfG \bfG_{\mathrm{win}^{(p)}}^\sharp(\regparam_p).
\end{align*}
In the derivation we will need the diagonal entries related to these resolution matrices. 
\begin{lemma}[\label{lem:diag}Diagonal entries of resolution matrices]
The diagonal entries of the matrices $\bfeye_m-\bfG \bfG^\sharp(\regparam_p)$ and $\bfeye_m-\bfG \bfG_{\mathrm{win}^{(p)}}^\sharp(\regparam_p)$ denoted by $\mu^{(p)}$ and $\nu^{(p)}$ are given by 
\begin{align*}
    \mu^{(p)}&= \frac{1}{m}\left(m-n + \sum_{j=1}^n(1- \Phi_{jj}(\regparam_p))\right) \text{ and }\\
    \nu^{(p)}&= \frac{1}{m}\left(m-n + \sum_{j=1}^n(1- (\Phi_{\mathrm{win}^{(p)}})_{jj}(\regparam_p))\right).
\end{align*}
\end{lemma}
\begin{proof}
We substitute back in for $\bfG$ and note that the filter matrices are of sizes $n\times n$
\begin{align*}
\bfG \bfG^\sharp(\regparam_p) = \bfC\bfDelta\inv{(\trans{\bfDelta}{\bfDelta} + \regparam_p^2 \trans{\bfLambda} {\bfLambda})}\trans{\bfDelta}{\bfC}^*=\bfC
\left[\begin{array}{cc}\bfPhi(\regparam_p) & 0 \\0 & 0\end{array}\right]{\bfC}^*.
\end{align*}
Also 
\begin{align*}
\bfeye_m-\bfG \bfG^\sharp(\regparam_p) = \bfC \left(\bfeye_m-\left[\begin{array}{cc}\bfPhi(\regparam_p) & 0 \\0 & 0\end{array}\right] \right){\bfC}^* = \bfC \left[\begin{array}{cc}\bfeye_n-\bfPhi(\regparam_p) & 0 \\0 & \bfeye_{m-n}\end{array}\right] {\bfC}^*,
\end{align*}
is circulant since the inner part of this product is diagonal. Therefore, its diagonal entries are equal and we can write $\diag(\bfeye_m-\bfG \bfG^\sharp(\regparam_p))= {\mu^{(p)}} \bfeye_m$. Further, since $C$ is unitary 
\begin{align}\label{eq:tracewinp}
\mathrm{trace}(\bfeye_m-\bfG \bfG^\sharp(\regparam_p))&= \sum_{j=1}^n (1-\Phi_{jj}(\regparam_p))+(m-n),
\end{align}
and 
\begin{align*}
    {\mu^{(p)}} = \frac{1}{m}\left(m-n + \sum_{j=1}^n(1- \Phi_{jj}(\regparam_p))\right).
\end{align*}
The result for the diagonal entries of $\bfG_{\mathrm{win}^{(p)}}(\regparam_p)$ follows similarly.
\end{proof}
Equipped with this result, we now proceed with the proof of \Cref{thm:WindowedGCV}.
\begin{proof}
To obtain $\G(\regparamVec)$ relies on deriving the Allen PRESS estimates \cite{GoHeWa} 
\begin{align*}
    F(\regparamVec) = \frac1m\sum_{k=1}^m\left( (\bfC\hat{\dVec})_k - (\bfG\yWin^{(k)})_k\right)^2,
\end{align*}
where $\yWin^{(k)}$ is the solution of \cref{eq:winTransform} with the $k^{\mathrm{th}}$ equation removed and $(\bm{a})_k$ indicates the $k^{\mathrm{th}}$ component of a vector $\bm{a}$. To find $\yWin^{(k)}$, we remove row $k$ from matrix $\bfG$ when forming \cref{eq:winTransform}. As in \cite{ChEaOl:11} this is accomplished using multiplication by matrix $\bfE_k=\bfeye_m-\bfe_{k}{\trans{\bfe}_k}$. But now, from the form of $\bfG$ we obtain $\bfE_k\bfG=\bfE_k \bfC\bfDelta=(\bfeye_m - \bfe_{k}\trans{\bfe}_k)\bfC\bfDelta$. But $\trans{\bfe}_k\bfC= \trans{\bfc}_k$ is the $k^{\mathrm{th}}$ row of $\bfC$. Because $\bfC$ is unitary this gives  $\trans{\bfG} \trans{\bfE}_k\bfE_k \bfG = \trans{\bfDelta}\bfDelta- \trans{\bfDelta}\bfc_k\trans{\bfc}_k\bfDelta $. Therefore, 
\begin{align*}
    \trans{\bfG}\trans{\bfE}_k\bfE_k{\bfG} + \regparam_p^2 \trans{\bfLambda} {\bfLambda} = \trans{\bfDelta}\bfDelta + \regparam_p^2 \trans{\bfLambda}{\bfLambda} - \trans{\bfDelta}\bfc_k\trans{\bfc}_k \bfDelta= \bfH(\regparam_p) -{\bfh}_k\trans{\bfh}_k,
\end{align*}
which defines diagonal matrix $\bfH(\regparam_p)$ and ${\bfh}_k=\trans{\bfDelta}\bm{\bfc}_k = \trans{\bfDelta}\bfC^*\bfe_k$. This is a rank one update for the matrix $\bfH(\regparam_p)$ and  the Sherman-Morrison formula \cite[Equation 2.1.5]{GolubVanLoan2013} gives 
\begin{align*}
    \inv{(\bfH(\regparam_p) -{\bfh}_k\trans{\bfh}_k)} = \inv{\bfH(\regparam_p)} +\frac{1}{\beta_k} \inv{\bfH(\regparam_p)} {\bfh}_k\trans{\bfh}_k \inv{\bfH(\regparam_p)},
\end{align*}
where $\beta_k=1- \trans{\bfh}_k \inv{\bfH(\regparam_p)} {\bfh}_k = \mu^{(p)}$ is independent of $k$ as already shown in \Cref{lem:diag}. Now $(\bfG\yWin^{(k)})_k$ is a sum over the $p$ windows of terms that involve 
\begin{align*}
\trans{\bfe}_k\bfC\bfDelta\inv{(\bfH(\regparam_p) -{\bfh}_k\trans{\bfh}_k)} &= \trans{\bfh}_k  \left(\inv{\bfH(\regparam_p)} +\frac{1}{\mu^{(p)}} \inv{\bfH(\regparam_p)} {\bfh}_k\trans{\bfh}_k \inv{\bfH(\regparam_p)}\right)\\
&=   \left(1 +\frac{1}{\mu^{(p)}} \trans{\bfh}_k\inv{\bfH(\regparam_p)} {\bfh}_k\right)\trans{\bfh}_k \inv{\bfH(\regparam_p)} =\frac{\trans{\bfh}_k \inv{\bfH(\regparam_p)}}{\mu^{(p)}} \\
&=\frac{\trans{\bfh}_k}{\mu^{(p)}}\inv{(\trans{\bfDelta}\bfDelta + \regparam_p^2 \trans{\bfLambda}\bfLambda)}.
\end{align*}
Therefore, 
\begin{align*}
    (\bfG\yWin^{(k)})_k &= \trans{\bfe}_k \sum_{p=1}^P \frac{1}{\mu^{(p)}}\bfC \bfDelta\inv{(\trans{\bfDelta}\bfDelta + \regparam_p^2 \trans{\bfLambda}\bfLambda)} \bfW^{(p)}\trans{\bfDelta} \bfC^* \bfE_k\bfC{\hat{\dVec}},\\
    &=\trans{\bfe}_k \bfC\left(\sum_{p=1}^P \frac{1}{\mu^{(p)}}
    \left[\begin{array}{cc}\bfPhi_{\mathrm{win}^{(p)}}(\regparam_p) & 0 \\0 & 0\end{array}\right]
      - \frac{1-\nu^{(p)}}{\mu^{(p)}}\bfeye_m\right)\hat{\dVec},
\end{align*}
and 
\begin{align*}
 (\bfC\hat{\dVec})_k - (\bfG\yWin^{(k)})_k = \trans{\bfe}_k \bfC \left(\bfeye_m-  \left(\sum_{p=1}^P \frac{1}{\mu^{(p)}}
    \left[\begin{array}{cc}\bfPhi_{\mathrm{win}^{(p)}}(\regparam_p) & 0 \\0 & 0\end{array}\right]
     + \frac{1-\nu^{(p)}}{\mu^{(p)}}\bfeye_m\right)\right)\hat{\dVec}.
\end{align*}
Finally, 
\begin{align*}
    F(\regparamVec) &= \frac{1}{m}\left\| \left(\bfeye_m-  \left(\sum_{p=1}^P \frac{1}{\mu^{(p)}}
    \left[\begin{array}{cc}\bfPhi_{\mathrm{win}^{(p)}}(\regparam_p) & 0 \\0 & 0\end{array}\right]
     + \frac{1-\nu^{(p)}}{\mu^{(p)}}\bfeye_m\right)\right)\hat{\dVec}\right\|^2 \\
     &=\frac{1}{m} \left(\sum_{j=q^*+1}^m \left(1 + \left(\sum_{p=1}^P\frac{1-\nu^{(p)}}{\mu^{(p)}}\right)\right)^2\hat{d}_j^2 + \right.\\
     & \left.\sum_{j=1}^{q^*}\left( 1 +\left(\sum_{p=1}^P\frac{1-\nu^{(p)}}{\mu^{(p)}}\right) -\left(\sum_{p=1}^P \frac{1}{\mu^{(p)}}
     \frac{\gamma_j^2w_j^{(p)}}{\gamma_j^2+\regparam_p^2}\right) \right)^2\hat{d}_j^2  \right).
\end{align*}
\end{proof}
We note that this result is consistent with the derivation using the standard form when $q^*=q<n$ \cite{Byrne} and with the result in \cite[Theorem 3.2]{ChEaOl:11}, when applied for the  standard Tikhonov regularization.

\section{Algebraic Results}\label{app:algebra}
Derivations for the terms needed to evaluate the regularization functions are provided. It is immediate that the common feature of the standard methods \cref{eq:UPRE,eq:GCV}, and their extensions,  is the need to evaluate $\|\rVec \|^2_2$ and $\trace(\A)$. We present the results for the windowed formulations, first focusing on the residual and then the trace, and in all cases assuming $m\ge n$. The results rely on the structure of  the filter matrix $\bfPhi_{\mathtt{win}}$. First we introduce 
$\bfPsi=\bfeye_n-\bfPhi$, with entries
\begin{align}\label{eq:Psientries}
    \Psi_{jj} =  \begin{cases} 1 & j=1:\ell \quad  (\delta_j=0)\\
    \frac{\alpha^2 \lambda_j^2}{\delta_j^2+\alpha^2 \lambda_j^2} & j=\ell+1:q^* \\
    0 & j=q^*:n \quad  (\lambda_j=0).
    \end{cases}
\end{align}
\begin{lemma}[Components of windowed data resolution matrix\label{lem:datacomp}]
\begin{align}\label{eq:datares}
   \bfeye_m - \bfA \bfA^\sharp_{\mathtt{win}}    &=  \bfU \left(  \left[ \begin{array}{cc}  \sum_{p=1}^P \bfW^{(p)} \Psi(\regparam_p) & 0 \\ 0 & \bfeye_{m-n}\end{array} \right]  
    \right) \trans{\bfU}.
\end{align}
\end{lemma}
\begin{proof}
Using the GSVD  and \cref{eq:AwinAwnsharp} 
\begin{align}
  \bfeye_m - \bfA \bfA^\sharp_{\mathtt{win}} & =  \bfeye_m- \bfU\left( \sum_{p=1}^P \bfW^{(p)} \Delta  \inv{(\trans{\Delta}\Delta +\regparam_p^2 \trans{\Lambda}\Lambda )}  \trans{\Delta} \nonumber
    \right) \trans{\bfU} \\
    &= \bfeye_m - \bfU \left(  \left[ \begin{array}{cc} \sum_{p=1}^P \Phi_{\mathrm{win}^{(p)}} & 0 \\ 0 & 0\end{array} \right]  
    \right)\trans{\bfU}. \label{eq:lemdatacomp2}
\end{align}
But now examining the components we have 
\begin{align}\label{eq:phientriesindr}
    1- \sum_{p=1}^{P}  w_j^{(p)} \frac{\delta_j^2} {\delta_j^2+\alpha_p^2\lambda_j^2}  = \begin{cases}  1 & j=1:\ell \\
      \sum_{p=1}^{P} w_j^{(p)}\frac{\regparam_p^2 \lambda_j^2}{\delta_j^2+\regparam_p^2 \lambda_j^2}& j=\ell+1:q^*\\
    0 & j=q^*:n. \end{cases}
\end{align}
Therefore, the weighted filter matrix in \cref{eq:datares} follows by comparing \cref{eq:phientriesindr} with \cref{eq:Psientries} and by noting in addition that the entries are $1$ for $j=n+1:m$.  
\end{proof}
\Cref{lem:datacomp} also applies for the complex spectral decomposition replacing $\trans{\bfU}$ by $\bfU^*$ for   unitary $\bfU$ and with entries $\delta_j^2$ and $\lambda_j^2$ replaced by $|\delta_j|^2$ and $|\lambda_j|^2$, respectively. The next results follow immediately from \Cref{lem:datacomp} and are given with minimal verification or without proof. 
\begin{lemma}[Norm of windowed residual\label{lem:normres}]
For the windowed residual   given by $\rWin(\regparamVec) = \bfA \xWin(\regparamVec) - \dVec$,  
\begin{align*}
    \|\rWin(\regparamVec)\|_2^2  &= \sum_{j=1}^{q^*} \left(\sum_{p=1}^Pw_j^{(p)}\Psi_{jj}(\regparam_p)  \hat{d}_j\right)^2+ \sum_{j=n+1}^m \hat{d}_j^2. 
\end{align*}
For non-overlapping windows with $w_j^{(p)}=1$ for  $j\in \mathtt{win}^{(p)}$ \begin{align*}
    \|\rWin(\regparamVec)\|_2^2  & = \sum_{p=1}^P \sum_{j\in\mathtt{win}^{(p)} } \left(\Psi_{jj}(\regparam_p)  \hat{d}_j\right)^2 + \sum_{j=n+1}^m \hat{d}_j^2 = \sum_{p=1}^P \|\hat{\rVec}_{\mathtt{win}}^{(p)}(\regparam_p)\|_2^2 + \sum_{j=n+1}^m \hat{d}_j^2,
\end{align*}
using the definition for the windowed residual based on the spectral domain given by 
$\hat{\rVec}_{\mathtt{win}}^{(p)}(\regparam_p) = \trans{\bfU}\bfA \xVec^{(p)}(\regparam_p) - \bfW^{(p)} \trans{\bfU} \dVec$ .  
\end{lemma}
\begin{proof}
The general result is immediate. For the non-overlapping case using the definition for the spectral residual with the definition for the windowed solution we have 
\begin{align}
\|\hat{\rVec}_{\mathtt{win}}^{(p)}(\regparam_p) \|^2&= \|\trans{\bfU}\bfA \xVec^{(p)}(\regparam_p) - \bfW^{(p)} \trans{\bfU} \dVec\|_2^2 \\
&= \| (\bfDelta \bfPhi(\regparam_p) \pinv{\bfDelta}- \bfeye_m ) \bfW^{(p)} \hat{\dVec} \|^2 = \sum_{j\in\mathtt{win}^{(p)} }\left(\Psi_{jj}(\regparam_p)  \hat{d}_j\right)^2, 
\end{align} 
and the result holds. 
\end{proof}
\begin{lemma}[Windowed Trace\label{lem:tracewin}]
\begin{align*}
    \trace( \bfA \bfA^\sharp_{\mathtt{win}}) &=  (n-q^*)+ \sum_{j=\ell+1}^{q^*} \sum_{p=1}^P w_j^{(p)} \Phi_{jj}(\regparam_p).
    \end{align*}
    For non-overlapping windows
    \begin{align*}
        \trace( \bfA \bfA^\sharp_{\mathtt{win}})&= (n-q^*)+\sum_{p=1}^P\sum_{j\in\mathtt{win}^{(p)} } \Phi_{jj}(\regparam_p). 
\end{align*}
\end{lemma}

\section{The GSVD and the  \texorpdfstring{$\mathtt{DCT}$}{DCT}\label{app:DCT}}
We outline the proof of \Cref{thm:Diag to GSVD}. \inserted{
\begin{proof}
We begin by setting $\trans{\bfC}\widetilde{\bfDelta}{\bfC} = \bfU\bfDelta\trans{\bfZ}$ and rearranging terms to obtain
$\bfDelta = \trans{\bfU}\trans{\bfC}\widetilde{\bfDelta}{\bfC}\invTrans{\bfZ}$. Doing the same for $\widetilde{\bfLambda}$,  and using $\trans{\bfDelta}\bfDelta + \trans{\bfLambda}\bfLambda = \bfeye_n$, we have 
\begin{align*}
	\inv{\bfZ}\trans{\bfC}\trans{\widetilde{\bfDelta}}\widetilde{\bfDelta}{\bfC}\invTrans{Z} + \inv{\bfZ}\trans{\bfC}\trans{\widetilde{\bfLambda}}\widetilde{\bfLambda}{\bfC}\invTrans{Z} &= \bfeye_n, 
	\end{align*}
from which we have 
\begin{align*} 
	\trans{\widetilde{\bfDelta}}\widetilde{\bfDelta} + \trans{\widetilde{\bfLambda}}\widetilde{\bfLambda} &= \bfC\bfZ\trans{\bfZ}\trans{\bfC}. \label{eq:Diag to GSVD}
\end{align*}
Since $\nullspace(A) ~\cap~ \nullspace(L) = \{\zeroVec\}$, the matrix on the left is diagonal with positive entries. Thus, we can form the (positive) square root $S = \sqrt{\trans{\widetilde{\bfDelta}}\widetilde{\bfDelta} + \trans{\widetilde{\bfLambda}}\widetilde{\bfLambda}}$ which is also diagonal with positive entries. Using $\bfS$ we can write $\bfS^2 = \bfS\trans{\bfS} = \bfC\bfZ\trans{\bfZ}\trans{\bfC}$, which implies that we can set $\bfZ = \trans{\bfC}\bfS$; $\bfS$ is invertible but not necessarily orthogonal, so $\bfZ$ is only invertible. Using the transpose $\trans{\bfZ} = \bfS\bfC$ and inverse $\inv{\bfS}$, we can then write
\begin{align*}
	\bfA &= \trans{\bfC}\widetilde{\bfDelta}\inv{\bfS}\bfS\bfC = \trans{\bfC}\widetilde{\bfDelta}\inv{\bfS}\trans{\bfZ} \text{ and }
	\bfL =  \trans{\bfC}\widetilde{\bfLambda}\inv{\bfS}\bfS\bfC = \trans{\bfC}\widetilde{\bfLambda}\inv{\bfS}\trans{\bfZ}.
\end{align*}
The last step is to reorder the elements of $\widetilde{\bfDelta}\inv{\bfS}$ and $\widetilde{\bfLambda}\inv{\bfS}$; fortunately they have the opposite order regardless of the order of the elements of $\widetilde{\bfDelta}$ and $\widetilde{\bfLambda}$. Therefore we can use a permutation matrix $\bfP$ so that $\bfDelta = \trans{\bfP}\widetilde{\bfDelta}\inv{\bfS}{\bfP}$ and $\bfLambda = \trans{\bfP}\widetilde{\bfLambda}\inv{\bfS}{\bfP}$ have the desired ordering. Since permutation matrices are orthogonal, we finally obtain the GSVD:
\begin{align*}
	\bfA &= \trans{\bfC}\widetilde{\bfDelta}\inv{\bfS}\trans{\bfZ} = \trans{\bfC}\bfP\trans{\bfP}\widetilde{\bfDelta}\inv{\bfS}{\bfP}\trans{\bfP}\trans{\bfZ} = \bfU\bfDelta\trans{\bfX}, \\
	\bfL &= \trans{\bfC}\widetilde{\bfLambda}\inv{\bfS}\trans{\bfZ} = \trans{\bfC}\bfP\trans{\bfP}\widetilde{\bfLambda}\inv{\bfS}{\bfP}\trans{\bfP}\trans{\bfZ} = \bfU\bfLambda\trans{\bfX},
\end{align*}
where $\bfU = \trans{\bfC}\bfP$ is orthogonal and $\bfX = \bfZ\bfP$ is invertible. Now it follows that  $\trans{\bfDelta}\bfDelta + \trans{\bfLambda}\bfLambda = \bfeye_n$ because 
\begin{align*}
	\trans{\bfDelta}\bfDelta + \trans{\bfLambda}\bfLambda &= \trans{\bfP}\invTrans{\bfS}\trans{\widetilde{\bfDelta}}\widetilde{\bfDelta}\inv{\bfS}\bfP + \trans{\bfP}\invTrans{\bfS}\trans{\widetilde{\bfLambda}}\widetilde{\bfLambda}\inv{\bfS}\bfP \\
	&= \trans{\bfP}\invTrans{\bfS}\left(\trans{\widetilde{\bfDelta}}\widetilde{\bfDelta} + \trans{\widetilde{\bfLambda}}\widetilde{\bfLambda}\right)\inv{\bfS}\bfP \\
	&= \trans{\bfP}\inv{\left(\bfS^2\right)}\left(\trans{\widetilde{\bfDelta}}\widetilde{\bfDelta} + \trans{\widetilde{\bfLambda}}\widetilde{\bfLambda}\right)\bfP \\
	&= \trans{\bfP}\inv{\left(\trans{\widetilde{\bfDelta}}\widetilde{\bfDelta} + \trans{\widetilde{\bfLambda}}\widetilde{\bfLambda}\right)}\left(\trans{\widetilde{\bfDelta}}\widetilde{\bfDelta} + \trans{\widetilde{\bfLambda}}\widetilde{\bfLambda}\right)\bfP = \bfeye_n.
\end{align*}
\end{proof}
}

The condition  $\nullspace(\bfA) ~\cap~ \nullspace(\bfL) = \{\zeroVec\}$ in  \Cref{thm:Diag to GSVD} is necessary for the property $\trans{\bfDelta}\bfDelta + \trans{\bfLambda}\bfLambda = \bfeye_n$.  \edited{While this identity is not used explicitly in any of the derivations used for finding the $\mathtt{UPRE}$ and $\mathtt{GCV}$ functions, it is standard that the mutual GSVD  decomposition is obtained with this condition imposed when it is calculated using the CS decomposition, and it is therefore relevant to show that the condition still holds for the given $\mathtt{DCT}$ case \cite{GoLo:96,ABT}.} 
As an alternative, we could require that $\trans{\bfDelta}\bfDelta + \trans{\bfLambda}\bfLambda = \widetilde{\bfeye}_n$ where $\widetilde{\bfeye}_n$ is a modified identity matrix that has some zero diagonal elements. This generalization, as well as a conversion involving Kronecker products, is described  in \cite[Chapter 2]{Byrne}. Furthermore, equipped with \Cref{thm:Diag to GSVD}, it is clear that the simplifications for the parameter selection methods using the GSVD can be rewritten in terms of the $\mathtt{DCT}$ simultaneous decomposition, which is particularly relevant for the efficient solution of two-dimensional problems. A similar approach can be used to convert the Kronecker product $\mathtt{DCT}$ for $2$D problems to a GSVD. 

\end{document}